\documentclass{amsart}
\usepackage{amssymb,latexsym}
\usepackage{amsrefs}
\usepackage{hyperref}
\usepackage[latin1]{inputenc}
\usepackage{subfigure}
\usepackage{amsmath,amssymb}
\usepackage{enumitem}
\usepackage{units}
\usepackage{graphicx}
\usepackage{a4wide}
\usepackage{color}
\usepackage{soul}

\usepackage[latin1]{inputenc}

\usepackage{tikz}

\usetikzlibrary{arrows}

\numberwithin{equation}{section} 
\newtheorem{theorem}{Theorem}[section]

\newtheorem{proposition}[theorem]{Proposition}
\newtheorem{lemma}[theorem]{Lemma} 

\newtheorem {corollary}[theorem]{Corollary}
\newtheorem{definition}[theorem]{Definition}

\newcommand{\dD}{{\mathcal D}}

\newcommand{\pP}{{\mathcal P}}
\newcommand{\rR}{{\mathcal R}}

\newcommand{\tT}{{\mathcal T}}

\newcommand{\tc}{\textcolor}
 
\newcommand{\ints}{{\mathbb Z}}

\definecolor{mycyan}{rgb}{0.4,0.8,1.0}

\definecolor{Green}{rgb}{0.2,0.7,0.2}

\title[Cluster complexes, partitions and Catalan arrangements]{Facets of the $m$-generalized Cluster complex and  
 regions in the $m$-extended Catalan arrangement of type $A_n$} 

\author[S. Fishel]{Susanna Fishel}
\address{School of Mathematical and Statistical Sciences, Arizona State University, Tempe, AZ 85287, USA}
\email{fishel@math.asu.edu}

\author[M. Kallipoliti]{Myrto Kallipoliti}
\address{Fak. f\"ur Mathematik, Universit\"at Wien, Garnisongasse 3, 1090 Wien, Austria}
\email{myrto.kallipoliti@univie.ac.at}

\author[E. Tzanaki]{Eleni Tzanaki}
\address{Department of Applied Mathematics, University of Crete, 71409 Heraklion, Crete, Greece}
\email{etzanaki@tem.uoc.gr}

\begin{document}

\begin{abstract}
In this paper we present a bijection $\omega_n$ between 
two well known families of Catalan objects: 
the set of facets of the $m$-generalized cluster complex $\Delta^m(A_n)$  
and the set of dominant regions in the $m$-Catalan arrangement ${\rm Cat}^m(A_n)$, 
where $m\in\mathbb{N}_{>0}$.  
In particular, $\omega_n$ bijects the facets containing the negative simple root $-\alpha$ to 
dominant regions having  the hyperplane 
$\{v\in V\mid\left\langle v,\alpha \right\rangle=m\}$ as separating wall. 
As a result, $\omega_n$ restricts to a bijection between 
the set of facets of the positive part of $\Delta^m(A_n)$ and the 
set of bounded dominant regions in ${\rm Cat}^m(A_n)$. 
The map $\omega_n$ is a composition of two bijections in which 
integer partitions in an $m$-staircase shape come into play. 

\end{abstract}

\maketitle
\thispagestyle{empty}

\setcounter{tocdepth}{1}

%\tableofcontents

%in terms of  integer partitions and obtain a result about regions in ${\rm Cat}^m(C_n)$ as  a special case.  
%Throughout  the paper, we use the characterization of the  facets of $\Delta^m(\Phi)$  in terms of polygon dissections, given by
%S.~Fomin and N.~Reading. 
 
%!TEX root = type_A.tex
\section{Introduction}
\label{sec:intro}
Fomin and Zelevinsky defined the \emph{cluster complex}
$\Delta(\Phi)$, a pure simplicial complex associated to every finite root
system $\Phi$ \cite{fz-ysga-03}, as a tool for the study of finite
cluster algebras \cite{fz-caf-02}. 
The ground set of $\Delta(\Phi)$ is the set of \emph{almost positive roots}, 
$\Phi_{\geq -1}$, which consists of one copy of each negative and positive simple root. 
Fomin and Reading \cite{fr-gcccc-05} generalized it to the \emph{$m$-cluster complex}
$\Delta^m(\Phi)$, a pure simplicial complex defined for any finite
(possibly reducible) Coxeter group and nonnegative integer
$m$. The ground set of $\Delta^m(\Phi)$ is the set of \emph{colored almost positive roots}, 
$\Phi^m_{\geq -1}$, which consists of one copy of each negative simple root together 
with $m$ copies of the each positive root. 
The complex $\Delta^m(\Phi)$ coincides with $\Delta(\Phi)$ when $m=1$.  The
generalized cluster complex contains a natural subcomplex, called the
positive part of $\Delta^m(\Phi)$ and denoted by $\Delta^m_+(\Phi)$.
It consists of faces of $\Delta^m(\Phi)$ that do not contain negative simple roots. 

When $\Phi$ is irreducible, the number of facets of $\Delta^m(\Phi)$
and $\Delta^m_+(\Phi)$ is equal to the {\it $m$-Catalan} number
$N(\Phi,m)=\prod_{i=1}^n \frac{e_i+mh+1}{e_i+1}$ and
{\it positive $m$-Catalan number} $N_+(\Phi,m) =
\prod_{i=1}^n \frac{e_i+mh-1}{e_i+1}$ respectively, where $n$ is the
rank, $h$ is the Coxeter number and $e_i$ are the exponents of $\Phi$.
In the special case where $\Phi=A_{n-1}$, we obtain the well-known 
\emph{Fuss-Catalan numbers}
\[N(\Phi,m)=\frac{1}{mn+1}{(m+1)n\choose n}.\]
The Fuss-Catalan numbers count an enormous number of combinatorial
objects. For instance, they count the number of $m$-Dyck (or $m$-ballot) paths of size $n$. 
These are lattice paths from $(0,0)$ to $(mn,n)$ using north steps $(0,1)$ 
and east steps $(1,0)$ which do not go below the line 
$y=\frac{1}{m}x$. Equivalently, they count the number of integer partitions whose 
Young diagram is contained in the $m$-staircase 
shape  defined by the partition $(m(n-1),m(n-2),\dots,m)$.
In this paper we use such integer partitions as a way to encode the facets 
of the generalized cluster complex $\Delta^m(A_{n-1})$,  
as well as the dominant regions of the extended Catalan arrangement 
of type $A_{n-1}$. 

The dominant regions of type $A_n$ constitute the second object of our study. 
We now give a brief description of the dominant regions in the general case of a finite crystallographic root system $\Phi$. 
For more information on this topic, we refer the reader to \cite{ath-gcn-04, ath-rgcn-05, atz-pc-06}. 
Let $\Phi$ be a finite crystallographic root system with set of
positive roots $\Phi_{>0}$ and let $V$ be the Euclidean space spanned by the set
$\Phi_{>0}$ with inner product $\left\langle \cdot,\cdot\right\rangle$. 
The {\it $m$-Catalan arrangement} ${\rm Cat}^m(\Phi)$ is
the collection of hyperplanes $\{H_{\alpha,k} \mid \alpha \in \Phi_{>0},\ 
0 \leq k \le m \}$, where $H_{\alpha,k} = \{ v \in V \mid \left\langle
v,\alpha \right\rangle = k \}$ for $\alpha \in \Phi$ and $k \in
\ints$. The hyperplanes of
${\rm Cat}^m(\Phi)$ dissect $V$ into {\em regions}. The {\it dominant
chamber} of $V$ is the intersection $\bigcap_{\alpha \in \Phi_{>0}} \{ v
\in V \mid \left\langle v,\alpha \right\rangle \ge 0 \}$, and is
also referred to as the \emph{fundamental chamber} in the literature.
Every region contained in the dominant chamber is called a {\it
  dominant region.}  Note that the dominant regions in the $m$-Catalan
arrangement are the same as the dominant regions in the $m$-Shi
arrangement \cite{sh-klccawg-86,st-hapfti-98}. The {\em bounded}
regions are those which only contain $v\in V$ satisfying
$0<\left\langle v,\alpha\right\rangle<m$ for all simple roots $\alpha$
of $\Phi$.  When $\Phi$ is an irreducible crystallographic root
system, the number of dominant and bounded dominant regions in ${\rm
  Cat}^m(\Phi)$ is again $N(\Phi,m)$ and $N_+(\Phi,m)$ respectively. 

Summarizing, for every finite crystallographic root system we have the
same number of facets of $\Delta^m(\Phi)$ as we have dominant regions
in ${\rm Cat}^m(\Phi)$.  Moreover, the number of facets of the
positive cluster complex $\Delta^m_+(\Phi)$ is equal to the number of
bounded dominant regions in the $m$-Catalan arrangement ${\rm
  Cat}^m(\Phi)$. We seek a bijection to explain these coincidences.
For $m=1$ there exist bijections between non-crossing partitions and
cluster complexes \cite{re-ccsencp-07} and between non-crossing
partitions and dominant regions in ${\rm Cat}^{1}(\Phi)$
\cite{ast-ub-11,cm-nnnnd-11,mam-nnnn-11}.  However, for $m \geq 2$ few
results that relate the above objects are known
\cite{brt-nnpc-09,krat-ftgcc-06,tz-fgcc-08}. 
To the best of our knowledge, there exists no bijection between dominant regions in ${\rm
  Cat}^m(\Phi)$ and facets of the cluster complex $\Delta^m(\Phi)$.

In this paper we close this gap in the case where $\Phi=A_n$. 
Before we state our main result we need to fix some notation. 
Let $\dD^m(A_n)$ (or $\dD^m_n$ for short) be  
the set of facets of the $m$-generalized cluster complex $\Delta^m(A_n)$ and let $\rR^m(A_n)$ (or $\rR^m_n$ for short) 
be the set of dominant regions in the $m$-Catalan arrangement ${\rm Cat}^{m}(A_n)$. 
We recall that a {\it wall} of a region $R$ is a hyperplane in ${\rm
  Cat}^m(\Phi)$ which supports a facet of $R$. We say that $H$ is a
{\it separating wall} of $R$, if the region $R$ and the origin lie in
different half-spaces relative to $H$. Finally, for $n\in\mathbb{N}$ we set $[n]=\{1,2,\dots,n\}$. 
Here we give a combinatorial proof of the following theorem. 

\begin{theorem}
\label{propA}
Let $\{\alpha_1,\alpha_2,\dots,\alpha_n\}$ be the set of simple roots of type $A_n$. 
There exists a bijection $\omega_n:\dD^m_n\to\rR^m_n$ with
the property that the negative simple root $-\alpha_i$ is 
contained in the facet $D\in\dD^m_n$ for $i\in[n]$, 
if and only if the hyperplane $H_{\alpha_i,m}$ is a
separating wall of the region $\omega_n(D)$.
\end{theorem}

We remark that the map $\omega_n$ of Theorem \ref{propA} restricts to a 
bijection between facets of $\Delta^m_+(A_n)$ and bounded regions in ${\rm Cat}^m(A_n)$. 
Even further, $\omega_n$ satisfies a certain refinement of this;  %our bijection should respect the
%following, which is proved in  Subsection  \ref{reducible}. \
it gives a combinatorial proof of the following observation, in case $\Phi=A_n$.

\begin{proposition}
\label{ourproperty} 
Let $\Phi$ be a finite crystallographic root system and let  
$\{\alpha_1,\alpha_2,\dots,\alpha_n\}$  be its set of simple roots. 
For any $J \subseteq [n]$,
the number of facets of $\Delta^m(\Phi)$ containing 
exactly the negative simple roots $-\alpha_i$ with $i \in J$,  is 
equal to the number of dominant regions in ${\rm Cat}^m(\Phi)$ 
with simple  separating walls $H_{\alpha_i,m}$ with $i\in J$.  
\end{proposition}

In order to construct the bijection $\omega_n$, 
we associate dominant regions of the $m$-Catalan arrangement ${\rm Cat}^{m}(A_n)$ to facets 
of the $m$-generalized cluster complex $\Delta^m(A_n)$ via certain integer partitions. 
More precisely, as mentioned above, we use as intermediaries integer partitions whose Young
diagram fits inside the diagram of the partition 
$(mn,m(n-1),\ldots,m)$ and we call them \emph{type-$A$} partitions of size $(n,m)$. 
We denote the set of such partitions by $\pP^m(A_n)$ (or $\pP^m_n$ for short). 
Partitions in $\pP^m_n$ 
can also be viewed as lattice paths from $(0,0)$ to
$(mn,n)$ using north and east steps which stay above the line 
$y=\frac{1}{m}x-1$, or equivalently $m$-Dyck, or $m$-ballot paths of size $n+1$.
Our main idea is to biject both facets of the cluster complex and 
dominant Catalan regions to type-$A$ partitions. 
A facet containing the negative simple root $-\alpha_i$ becomes a partition 
$(\lambda_1,\lambda_2,\dots,\lambda_n)$ whose
$i$-th part has maximum size, that is 
$\lambda_i=(n-i+1)m$. On the other hand, a dominant region with simple separating wall
$H_{\alpha_i,m}$ also bijects to a partition whose $i$-th 
part has maximum size. We remark that a 
maximum-sized part in a type-$A$ partition corresponds to a point where the lattice
path touches a line of the form $y=x/m+t$, where $t\in\mathbb{R}$, 
or where the $m$-Dyck path touches the line $y=\frac{1}{m}x$. 
Touch points (also called contacts) appear often in the literature. 
For instance, in \cite{bfp-11}, touch points are used to derive a recursion to count
the number of intervals in a generalization of the $m$-Tamari lattice and 
in \cite{hmz-10} sorting $1$-Dyck paths by touch points leads to a refinement of the shuffle conjecture. 
In this paper we prove combinatorially the following theorems, which eventually lead to the construction of $\omega_n$. 

\begin{theorem}
\label{all partitions}
Let $\{\alpha_1,\alpha_2,\dots,\alpha_n\}$ be the set of simple roots of type $A_n$. 
There exists a bijection $\varphi_n:\rR^m_n\to\pP^m_n$ with
the property that the hyperplane $H_{\alpha_i,m}$ is a separating wall 
of the region $R\in \rR^m_n$ for $i\in[n]$, 
if and only if the $i$-th part of the partition $\varphi_n(R)$ is equal to $(n-i+1)m$.
\end{theorem}

\begin{theorem}
\label{snakedisA}
Let $\{\alpha_1,\alpha_2,\dots,\alpha_n\}$ be the set of simple roots of type $A_n$. 
There exists a bijection $\psi_n:\dD^m_n\to \pP^m_n$ with
the property that the negative simple root $-\alpha_i$ is contained 
in the facet $D\in\dD^m_n$ for $i\in[n]$, 
if and only if the $i$-th part of the partition $\psi_n(D)$ is equal to $(n-i+1)$.
\end{theorem}

For our bijections we use a realization of the facets of the cluster complex in terms of
polygon dissections given by S.~Fomin and N.~Reading
\cite{fr-gcccc-05}. Also, we realize the regions in terms of certain
tableaux which we call \emph{type-$A$ Shi tableaux}.

This paper is structured as follows. In Section \ref{sec:prelim} we
provide the necessary background, fix notation and prove Proposition
\ref{ourproperty}, which constitutes the motivation for our work. 
Since the realization of the dominant regions in ${\rm Cat}^m_n$ as Shi tableaux will be crucial for our proofs, 
in Section \ref{shi} we describe these tableaux in detail.
In Section \ref{regions to facets} we prove Theorem 
\ref{all partitions}. Theorem \ref{snakedisA} is proved in Section 
\ref{facets to partitions}. We complete this paper with Section \ref{sec:concl}, where we prove Theorem \ref{propA} 
by composing the bijections of Theorems \ref{all partitions} and \ref{snakedisA}.
We also include an application of Theorem \ref{snakedisA}, 
where we rediscover the positive Catalan numbers of type $A_n$ bijectively. 
We therefore answer \cite[Remark 2]{atz-pc-06} in the case where $k$ is the rank of
the root system, and $\Phi=A_n$. %Finally, we state an open question. 

%!TEX root = type_A.tex
\section{Preliminaries}
\label{sec:prelim} 

\subsection{Root system of type $A_n$} 
\label{sec:roots}
Let  $\varepsilon_1,\varepsilon_2,\ldots,\varepsilon_{n+1}$ be the standard basis in $\mathbb{R}^{n+1}$.   
In type $A_n$, a standard choice of positive and simple roots are respectively, 
the sets $\{\alpha_{ij} \mid  1 \leq i\leq j \leq n\}$ and $\{\alpha_i \mid 1\leq i \leq n\}$, 
where $\alpha_{ij}:= \varepsilon_i-\varepsilon_{j+1}$ and $\alpha_i:=\varepsilon_i-\varepsilon_{i+1}$, 
for every $1\leq i\leq j\leq n$.  
The positive roots can be written in terms of
simple roots as
\[\alpha_{ij}=\alpha_i+\alpha_{i+1}+\cdots+\alpha_j,\ \mbox{for every}\  1\leq i\leq j\leq n.\] 
For instance, the positive roots of $A_2$ are the vectors: 
$\alpha_1=(1,-1,0), \alpha_2=(0,1,-1)$ and $\alpha_{12}=(1,0,-1)$.   

\subsection{Cluster complexes and polygon dissections}
\label{subsec:dis}
The cluster complex $\Delta^m(A_n)$ is a pure simplicial complex of dimension $n-1$ on the
ground set of \emph{colored almost positive roots}
which consists of $m$ ``colored'' copies of each positive root and one
copy of each negative simple root. This simplicial complex 
can be described in terms of \emph{polygon dissections}. 
We refer the reader to \cite{fr-gcccc-05} for background on the generalized cluster complex.

Let $P$ be a polygon with $m(n+1)+2$ vertices labeled from $0$ up to
$m(n+1)+1$ in counterclockwise order. An \emph{$m$-diagonal} in $P$
is a diagonal dissecting $P$ into a pair of polygons, where for each
polygon the number of vertices is congruent to $2$ modulo $m$. 
A collection of $n$ such $m$-diagonals is a {\em $(m+2)$-angulations} of $P$. 
We call such a $(m+2)$-angulation a \emph{type-$A$ (polygon) dissection} of size $(n,m)$, or 
\emph{$m$-dissection} for short.

The faces of $\Delta^m(A_n)$ of dimension $k$ 
are in bijection with dissections of $P$
having $k+1$ pairwise non crossing $m$-diagonals.  
The facets of $\Delta^m(A_n)$
correspond to $m$-dissections of $P$. 
For $1 \leq i \leq \frac{n+1}{2}$, the negative simple root $-\alpha_{2i-1}$ 
is identified with the diagonal of $P$ connecting the
vertex $(i-1)m$ to the vertex $(n+1-i)m+1$. 
For $1\leq i\leq\frac{n}{2}$, we identify the negative simple root $-\alpha_{2i}$ with
the diagonal of $P$ connecting the vertex $im$ to the vertex $(n+1-i)m+1$. 
We notice that the negative simple roots form a ``snake'' in $P$ (see Figure \ref{fig:snake1}). 

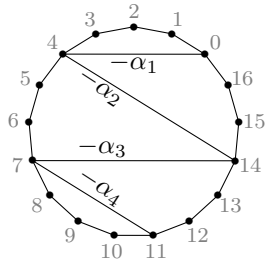
\begin{figure}[h]
\begin{tikzpicture}[scale=0.7]
\path (280.59: 2 cm) coordinate (A0); \path (301.76: 2 cm) coordinate  (A1); \path (322.93 : 2 cm) coordinate  (A2);\path (344.1 : 2 cm) coordinate  (A3);
\path (365.27 : 2 cm) coordinate  (A4); \path (26.44 : 2 cm) coordinate  (A5); \path (47.61 : 2 cm) coordinate  (A6); \path (68.78: 2 cm) coordinate  (A7);
\path (89.95 : 2 cm) coordinate  (A8); \path (111.12 : 2 cm) coordinate  (A9); \path (132.29 : 2 cm) coordinate  (A10);
\path (153.46: 2 cm) coordinate  (A11); \path (174.63: 2 cm) coordinate  (A12); \path (195.8 : 2 cm) coordinate  (A13);
\path (216.97 : 2 cm) coordinate  (A14); \path (238.14: 2 cm) coordinate  (A15);   \path (259.31  : 2 cm) coordinate  (A16);
\draw (A0)--(A1)--(A2)--(A3)--(A4)--(A5)--(A6)--(A7)--(A8)--(A9)--(A10)--(A11)--(A12)--(A13)--(A14)--(A15)--(A16)--(A0);
\node  at (A0){ \tiny $\bullet$ };  \node  at (A1){  \tiny $\bullet$};  \node  at (A2){  \tiny$\bullet$ };  \node  at (A3){ \tiny $\bullet$};
\node  at (A4){ \tiny $\bullet$}; \node  at (A5){  \tiny $\bullet$ };  \node  at (A6){  \tiny $\bullet$};  \node  at (A7){  \tiny $\bullet$ }; 
\node  at (A8){  \tiny $\bullet$}; \node  at (A9){ \tiny $\bullet$}; \node  at (A10){  \tiny $\bullet$ };  \node  at (A11){  \tiny $\bullet$};
\node  at (A12){ \tiny $\bullet$ };  \node  at (A13){ \tiny $\bullet$}; \node  at (A14){  \tiny $\bullet$}; \node  at (A15){ \tiny$\bullet$ };
\node  at (A16){ \tiny$ \bullet $ };  
\node at  (280.59: 2.3 cm) { \footnotesize \tc{black!50}{11} }; 
\node at (301.76: 2.3 cm) {\footnotesize \tc{black!50}{12}}; 
\node at  (322.93 : 2.3 cm) {\footnotesize\tc{black!50}{13}};  
\node at  (344.1 : 2.3 cm) {\footnotesize\tc{black!50}{14}};
\node at  (365.27 : 2.3 cm) {\footnotesize\tc{black!50}{15}}; 
\node at  (26.44 : 2.3 cm) {\footnotesize\tc{black!50}{16}}; 
\node at  (47.61 : 2.3 cm) {\footnotesize\tc{black!50}{0}}; 
\node at (68.78: 2.3 cm) {\footnotesize\tc{black!50}{1}};
\node at  (89.95 : 2.3  cm) {\footnotesize\tc{black!50}{2}}; 
\node at (111.12 : 2.3 cm) {\footnotesize\tc{black!50}{3}}; 
\node at  (132.29 : 2.3 cm){\footnotesize\tc{black!50}{4}}; 
\node at (153.46: 2.3 cm){\footnotesize\tc{black!50}{5}}; 
\node at (174.63: 2.3 cm){\footnotesize\tc{black!50}{6}}; 
\node at  (195.8 : 2.3 cm){\footnotesize\tc{black!50}{7}}; 
\node at  (216.97 : 2.3 cm) {\footnotesize\tc{black!50}{8}}; 
\node at  (238.14: 2.3 cm){\footnotesize\tc{black!50}{9}}; 
\node at  (259.31  : 2.3 cm) {\footnotesize\tc{black!50}{10}};   
\draw[black] (A0)--(A13); \draw[black] (A3)--(A13);\draw[black] (A6)--(A10); \draw[black] (A3)--(A10);
%\draw[black] (A0)--(A7); \draw[black] (A1)--(A8);\draw[black] (A2)--(A9); 
\node at (-0.7,-1) { \rotatebox{-33}{ \tc{black}{ $ -\alpha_4$}}}; 
\node at (-0.6,-0.3) { \tc{black}{$-\alpha_3$}};  \node at (-0.7,0.9) { \rotatebox{-33}{ \tc{black}{ $ -\alpha_2$}}}; 
\node at (0,1.3) { \tc{black}{$-\alpha_1$}};
\end{tikzpicture} 
\caption{The snake formed by the negative simple roots for $\Delta^3(A_4)$.}
%Also three copies of the root $\alpha_{24}= \alpha_2+\alpha_3+\alpha_4$ are depicted in $\Delta^3(A_4)$ (left).}
\label{fig:snake1}
\end{figure} 

For each pair $(i,j)$ with $1\leq i\leq j\leq n+1$ there are exactly $n$ many  $m$-diagonals
that intersect (in the interior) the diagonals
$-\alpha_i,-\alpha_{i+1},\ldots,-\alpha_j$ and no other diagonals in
the snake. 
The diagonal corresponding to the colored positive root $\alpha^k_{ij}$ 
is uniquely determined as the $k$-th in clockwise order
of the exactly $n$ $m$-diagonals which intersect
$-\alpha_i,-\alpha_{i+1},\ldots,-\alpha_j$. 
Under this identification, every $m$-diagonal in $P$ corresponds to an almost
positive root in $\Delta^m(A_n)$ \cite[Section 5.1]{fr-gcccc-05}. 
In this paper we denote by
$\dD^m(A_n)$ (or $\dD^m_n$ for short) the set of $m$-dissections of an $(m(n+1)+2)$-gon.

\begin{definition}
We say that $D\in \dD^m_n$ \emph{contains} the negative simple root
$-\alpha_i$, for some $i\in[n]$, if the diagonal
corresponding to $-\alpha_i$ is contained in $D$.
\end{definition}

\subsection{Integer partitions}
\label{partitions}
An \emph{integer partition} is a nonincreasing sequence
$\lambda=(\lambda_1,\lambda_2,\ldots,\lambda_n)$ of nonnegative
integers, called  parts. We identify a partition
$\lambda=(\lambda_1,\lambda_2,\ldots, \lambda_n)$ with its \emph{Young
  diagram}, the left-justified array of boxes with $\lambda_i$ boxes
in row $i$. The box in row $i$, column $j$, has
coordinates $(i,j)$. In what follows we describe the partitions we are interested in. 
Let $n$ and $m$ be positive integers. 
We denote by $\pP^m(A_n)$ (or $\pP^m_n$ for short) the set of partitions whose Young
  diagram fits into an $m$-staircase shape of size $n$ (see Figure \ref{fig:part}): 

  \[\pP^m_n=\{(\lambda_1,\lambda_2,\ldots,\lambda_n) \mid 0\leq\lambda_i\leq
  m(n-i+1)\}.\] 

\begin{figure}[h]
\begin{center}
  \begin{tikzpicture}[scale=0.4]
  \foreach \i in {0,...,11} 
  {\draw (\i,3) rectangle (1+\i,4); } 
  \foreach \i in {0,...,8} 
  {\draw (\i,2) rectangle (1+\i,3); } 
  \foreach \i in {0,...,5} 
  {\draw (\i,1) rectangle (1+\i,2); } 
  \foreach \i in {0,...,2} 
  {\draw (\i,0) rectangle (1+\i,1); }
\end{tikzpicture}
\end{center}
\caption{A $3$-staircase shape of size $4$.}
\label{fig:part}
\end{figure}
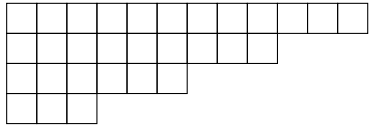

  The number of such partitions is
  $\frac{1}{m(n+1)+1}{(m+1)(n+1)\choose n+1}$.  Throughout this paper
  we will refer to partitions in $\pP^m_n$ as \emph{$(n,m)$-type $A$} partitions 
  (or just \emph{type-$A$} partitions, when there is no ambiguity).

\subsection{Dominant regions in the $m$-Catalan arrangement and Shi tableaux}
\label{subsec:drst}
Shi \cite{sh-nost-97} and others arrange the positive roots in diagrams
which encode the poset structure of the positive roots. We use this
idea to store coordinates describing the location of dominant regions
of the Catalan arrangement. We call the diagrams {\em Shi
tableaux}. 

Let $R$ be a dominant region in ${\rm Cat}^m(\Phi)$. 
By definition of the $m$-Catalan arrangement,  
for each $ \alpha \in \Phi_{>0}$  there exists an integer $0 \leq k_\alpha \leq m$ 
such that for all $x \in R$ we have
 $k_{\alpha} \leq \langle\alpha,x\rangle \leq k'_\alpha $, where 
 $$k'_\alpha =  \begin{cases}  k_\alpha +1 \;\;\; \mbox{ if }  0 \leq   k_\alpha\leq m-1,  \\
                      +\infty \;\;\;\;\;\; \mbox{  otherwise.} \end{cases} $$

\begin{definition}
\label{def:st} 
We define the \emph{Shi tableau for a dominant region} $R$ in 
${\rm Cat}^m(\Phi)$ to be the set of positive integers 
$\{k_\alpha \mid \alpha\in\Phi_{>0}\}$.  
\end{definition}  

Each Shi tableau is attached to a dominant region. 
More precisely, for each $ \alpha \in \Phi_{>0}$, the coordinate
$k_{\alpha}$ is the number of integer translations of the hyperplane
$H_{\alpha}$ which separate the region from the origin.

\subsection{Reducible crystallographic root systems and the proof of Proposition \ref{ourproperty}} 
\label{reducible} 
The definition of the cluster complex $\Delta^m(\Phi)$ as well as that
of the $m$-Catalan arrangement can be generalized to reducible
crystallographic root systems as follows. If $\Phi=\Phi_1 \times \cdots
\times \Phi_\ell$ is the factorization of $\Phi$ into irreducible root
systems, then $\Delta^m(\Phi)$ is defined as the simplicial join of
$\Delta^m(\Phi_i)$ with $ 1 \leq i \leq \ell$ 
\cite[Section 3]{fr-gcccc-05}, thus $N(\Phi,m)=\prod_{i=1}^\ell N(\Phi_i,m)$. 
Moreover, the set of facets of $\Delta^m(\Phi)$ containing no negative simple roots is equal to the
product $\prod_{i=1}^\ell N_+(\Phi_i,m)$ 
\cite[Corollary 12.1]{fr-gcccc-05}. 
On the other hand, the $m$-Catalan arrangement ${\rm Cat}^m(\Phi)$ 
generalizes straightforwardly in the case where $\Phi$ is reducible. 
Clearly, the number of regions and bounded regions is equal to 
$\prod_{i=1}^\ell N(\Phi_i,m)$ and
$\prod_{i=1}^\ell N_+(\Phi_i,m)$ respectively.
In what follows, we write $\Phi_I$ instead of $\Phi,$ where $I$ is an
index set in bijection with the set  $\Pi$ of simple roots. 
For $J\subseteq I$ we will denote by $\Phi_J$ the parabolic root system with
simple roots $\{\alpha_i \mid  i \in J\}$. Note that $\Phi_J $ may be
reducible, even in the case where $\Phi_I$ is irreducible.

\begin{proof}[Proof of Proposition \ref{ourproperty}]
From \cite[Theorem 3.7]{fr-gcccc-05} we deduce that the number of
facets of $\Delta^m(\Phi_I)$ having exactly the negative simple roots
$ -\alpha_i, i \in J$ is equal to the number of facets of $\Delta^m(
\Phi_{I \setminus J })$ having no negative simple roots.  From the
proof of \cite[Lemma 5.3]{atz-pc-06} we deduce that the number of
dominant regions in ${\rm Cat}^m(\Phi_I)$ with simple separating walls
$ H_{\alpha_i }, i \in J $ is equal to the number of bounded dominant
regions in $ {\rm Cat}^m(\Phi_{I \setminus J}). $ The result then
follows, since the number of facets of $\Delta^m( \Phi_{I \setminus J
})$ having no negative simple roots is equal to the number of bounded
dominant regions in $ {\rm Cat}^m(\Phi_{I \setminus J})$. 
\end{proof}

\noindent\emph{Notation}. For $n\in\mathbb{N}$ we denote by $[n]$ and $[n]_{\geq 0}$ the 
sets $\{1,2,\dots,n\}$ and $\{0,1,2,\dots,n\}$, respectively. 

%!TEX root = type_A.tex
\section{Shi tableaux}
\label{shi}

In this section we describe a way to represent the Shi tableaux for type-$A$ dominant regions.  
Let $R$ be a dominant region in $\rR^m_n$ and let $T$ be a staircase Young diagram of size $(n,n-1,\ldots,1)$. 
We arrange the coordinates $k_{\alpha_{ij}}$ corresponding to  $R$ in the diagram $T$ so that for each 
$1\leq i \leq j \leq n$, the integer $k_{\alpha_{i,j}}$ is placed in the box $(i,n-j+1)$. 
To simplify the notation, we write  $k_{i,j}$ instead of $k_{\alpha_{ij}}$  (see Figure \ref{fig:rootposetA}). 

\begin{figure}[h]
\begin{center}
\begin{tikzpicture}[scale=0.9]
%type A
\draw (3,1) rectangle (7,2);
\node at (3.5,1.5) {$k_{1,4}$};\node at (4.5,1.5){$k_{1,3}$};\node at (5.5,1.5){$k_{1,2}$};\node at (6.5,1.5){$k_{1,1}$}; 
\draw (3,0) rectangle (6,2);
\node at (3.5,0.5) {$k_{2,4}$};\node at (4.5,0.5){$k_{23}$};\node at (5.5,0.5){$k_{2,2}$};
\draw (3,-1) rectangle (5,2);
\node at (3.5,-0.5) {$k_{3,4}$};\node at (4.5,-0.5) {$k_{3,3}$};
\draw (3,-2) rectangle (4,2);
\node at (3.5,-1.5) {$k_{4,4}$};
\end{tikzpicture}
\caption{The distribution of the coordinates in a
4-staircase diagram for type $A_4$.}
\label{fig:rootposetA}
\end{center}
\end{figure}
%We call such tableau the \emph{type-$A$}  (see Figure \ref{fig:rootposetA}) Shi tableau of the region $R$. 
Throughout this paper, we identify each dominant regions of an $m$-Catalan arrangement $\rR^m_n$ with its Shi tableau. 

\

The entries of a Shi tableau satisfy certain conditions
\cite{ath-rgcn-05}. In order to describe them explicitly, we need to fix some notation.
Let $\mathtt{x}$ be a box in a Young diagram $T$. A {\it hook $\mathsf h$ of length
  $\ell$ on a box $ \mathtt  x$} of $T$ is an array of $\ell$ contiguous
boxes lying to the right (and same row) or below (and same column) of
$ \mathtt  x$ including the box $\mathtt  x$ itself.  The box $\mathtt x$ is called 
\emph{the corner} of the hook.  
For a Young diagram whose boxes are filled with
numbers, we denote by ${\mathsf e}(\mathsf h)$ the sum of the numbers at the
endpoints of the hook $\mathsf h$.
The following characterization for the entries of a Shi tableau can be deduced 
from results in \cite{ath-rgcn-05,sh-acawg-97}. 
A more detailed proof can be found in \cite[Lemma 2.5]{fvt-fsw-11}.

\begin{proposition} 
%\label{hook condition} 
Let $T=\{k_{i,j} \mid 1 \leq i \leq j \leq n \}$ be a set of integers in $[m]_0$ 
arranged in a staircase diagram of size $(n,n-1,\dots,1)$ so that 
 $k_{i,j}$ is placed in the box $(i,n-j+1)$. 
Then, $T$ is a type-$A$ Shi tableau of some region $R\in \rR^m_n$ if and only if 
for every $1\leq i < j \leq n$ and any hook ${\mathsf h}_{k_{i,j}}$ on $k_{i,j}$ of length $j-i+2$  we have        
\begin{equation}  
\label{acond}
    k_{i,j}=
    \begin{cases}{\mathsf e}({\mathsf h}_{k_{i,j}})+\delta, \hspace{1  cm} \mbox{if}\ {\mathsf e}( {\mathsf h}_{k_{i,j}}) < m \\
    m \hspace{2.5  cm} \mbox{otherwise,}  
    \end{cases} 
\end{equation} 
where $\delta\in\{0,1\}$. 
Equivalently, $T$ is a type-$A$ Shi tableau of some region $R\in \rR^m_n$ if and only if 
for every $1\leq i < j \leq n$ and $i\leq \ell < j$ we have 
\begin{equation}
\label{Ashi}
 k_{i,j} = \begin{cases} 
	   k_{i\ell} + k_{\ell+1,j }+\delta,  \hspace{1  cm} \mbox{if}\ k_{i\ell} + k_{\ell+1,j } < m   \\                    
           m \hspace{3.2  cm} \mbox{otherwise,}  
	    \end{cases}
            \end{equation} 
where $\delta\in\{0,1\}$.  
\end{proposition}

We say that the entry (or corner) $k_{i,j}$ of $T$ satisfies the Shi conditions, 
if  Equation (\ref{Ashi}) holds for all $i \leq \ell<j$. 
Furthermore, if Equation (\ref{Ashi}) holds for some fixed $\ell$, then we say that 
the triplet $\{k_{i,j},k_{i,\ell},k_{\ell+1,j}\} $ satisfies the Shi conditions. 

\begin{definition}
Let $T$ be a staircase Young diagram of size $(n,n-1,\ldots,1)$ 
and let $k_{i,j}$ be the entry of the box $(i,n-j+1)$. 
If  the Shi conditions hold for every $1\leq i\leq j\leq n$, we say that $T$ is a 
\emph{type-$A$} Shi tableau of size $(n,m)$. 
%\footnote{We use here the term \emph{type-$A$} Shi tableau, in order to distinguish it from a 
%tableau of a region in Catalan arrangement of other types \cite{fkt-12}.} 
\end{definition}

Figure \ref{fig:hypera} 
illustrates  the dominant regions of the hyperplane arrangements 
$\mbox{Cat}^3(A_2)$ together with their Shi tableaux. 

\begin{figure}[h]
\begin{center}
\begin{tikzpicture}[scale=1.4] 
%define the endpoints for the orthogonal hyperplanes
\path (0:6cm) coordinate (+O2); \path (0:-3cm) coordinate (-O2);  \node at (0:6.5cm) {$H_{\alpha_2,0}$};
\begin{scope}[yshift= 0.87 cm]
\path (0:6cm) coordinate (+T2); \path (0:-3cm) coordinate (-T2); \node at (0:6.5cm) {$H_{\alpha_2,1}$};
\end{scope} 
\begin{scope}[yshift= 1.74 cm]
\path (0:6cm) coordinate (+TT2); \path (0:-3cm) coordinate (-TT2); \node at (0:6.5cm) {$H_{\alpha_2,2}$};
\end{scope} 
\begin{scope}[yshift= 2.61 cm]
\path (0:6cm) coordinate (+TTT2); \path (0:-3cm) coordinate (-TTT2); \node at (0:6.5cm) {$H_{\alpha_2,3}$};
\end{scope}
\path (120:4cm) coordinate (+O12); \path (120:-1.2cm) coordinate (-O12); \node at (120:4.6 cm) { \rotatebox{-60}{$H_{\alpha_{12},0}$ }  };
\begin{scope}[xshift= 1cm] 
\path (120:4cm) coordinate (+T12); \path (120:-1.2cm) coordinate (-T12); \node at (120:4.6 cm) {\rotatebox{-60}{ $H_{\alpha_{12},1} $ }  };
\end{scope} 
\begin{scope}[xshift= 2cm] 
\path (120:4cm) coordinate (+TT12); \path (120:-1.2cm) coordinate (-TT12); \node at (120:4.6 cm) {\rotatebox{-60}{ $H_{\alpha_{12},2} $ }  };
\end{scope} 
\begin{scope}[xshift= 3cm] 
\path (120:4cm) coordinate (+TTT12); \path (120:-1.2cm) coordinate (-TTT12); \node at (120:4.6 cm) {\rotatebox{-60}{$H_{\alpha_{12},3} $ }  };
\end{scope} 
\path (60:5cm) coordinate (+O1); \path (60:-1.3cm) coordinate (-O1); \node at (60:5.5cm) { \rotatebox{60}{$H_{\alpha_1,0}$}};
\begin{scope}[xshift=1 cm]
\path (60:5cm) coordinate (+T1); \path (60:-1.3cm) coordinate (-T1);  \node at (60:5.5cm) { \rotatebox{60}{$H_{\alpha_1,1}$}}; 
\end{scope} 
\begin{scope}[xshift=2 cm]
\path (60:5cm) coordinate (+TT1); \path (60:-1.3 cm) coordinate (-TT1);  \node at (60:5.5cm) { \rotatebox{60}{$H_{\alpha_1,2}$}};  
\end{scope}
\begin{scope}[xshift=3 cm]
\path (60:5cm) coordinate (+TTT1); \path (60:-1.3 cm) coordinate (-TTT1);  \node at (60:5.5cm) { \rotatebox{60}{$H_{\alpha_1,3}$}}; 
\end{scope}

          % the orthogonal hyperplanes
\draw[line width= 2pt] (-O2) -- (+O2);
\draw[line width= 2pt] (-O12) -- (+O12);
\draw[line width= 2pt] (-O1) -- (+O1);
\draw[line width= 0.7pt] (-T12) -- (+T12);
\draw[line width= 0.7pt] (-T2) -- (+T2);
\draw[line width= 0.7pt] (-T1) -- (+T1);
\draw[line width= 0.7pt] (-TT2) -- (+TT2);
\draw[line width= 0.7pt] (-TT12) -- (+TT12);
\draw[line width= 0.7pt] (-TT1) -- (+TT1);
\draw[line width= 0.7pt] (-TTT2) -- (+TTT2);
\draw[line width= 0.7pt] (-TTT12) -- (+TTT12);
\draw[line width= 0.7pt] (-TTT1) -- (+TTT1);

\begin{scope}[scale=0.2, xshift=1.5 cm, yshift=0.3 cm]
\draw (0,0) rectangle (1,2); \draw (0,1) rectangle (2,2);
\node at (.5,1.45) {\footnotesize 0};  
\node at (1.5,1.45) {\footnotesize 0}; 
\node at (.5,0.45) {\footnotesize 0}; 
\end{scope}
  
\begin{scope}[scale=0.2, xshift=4.3 cm, yshift=1.7 cm]
\draw (0,0) rectangle (1,2); \draw (0,1) rectangle (2,2);
\node at (.5,1.45) {\footnotesize 1};  
\node at (1.5,1.45) {\footnotesize 0}; 
\node at (.5,0.45) {\footnotesize 0}; 
\end{scope}
  
\begin{scope}[scale=0.2, xshift=4 cm, yshift=4.7 cm]
\draw (0,0) rectangle (1,2); \draw (0,1) rectangle (2,2);
\node at (.5,1.45) {\footnotesize 1}; 
\node at (1.5,1.45) {\footnotesize 0}; 
\node at (.5,0.45) {\footnotesize 1}; 
\end{scope}

\begin{scope}[scale=0.2, xshift=6.7 cm, yshift=6 cm]
\draw (0,0) rectangle (1,2); \draw (0,1) rectangle (2,2);
\node at (.5,1.45) {\footnotesize 2};  
\node at (1.5,1.45) {\footnotesize 0}; 
\node at (.5,0.45) {\footnotesize 1}; 
\end{scope}
  
\begin{scope}[scale=0.2, xshift=9 cm, yshift=4.5 cm]
\draw (0,0) rectangle (1,2); \draw (0,1) rectangle (2,2);
\node at (.5,1.45) {\footnotesize 2}; 
\node at (1.5,1.45) {\footnotesize 1}; 
\node at (.5,0.45) {\footnotesize 1}; 
\end{scope}
  
\begin{scope}[scale=0.2, xshift=6.5 cm, yshift=0.3 cm]
\draw (0,0) rectangle (1,2); \draw (0,1) rectangle (2,2);
\node at (.5,1.45) {\footnotesize 1}; 
\node at (1.5,1.45) {\footnotesize 1}; 
\node at (.5,0.45) {\footnotesize 0}; 
\end{scope}
  
\begin{scope}[scale=0.2, xshift=9.3 cm, yshift=1.7 cm]
\draw (0,0) rectangle (1,2); \draw (0,1) rectangle (2,2);
\node at (.5,1.45) {\footnotesize 2}; 
\node at (1.5,1.45) {\footnotesize 1}; 
\node at (.5,0.45) {\footnotesize 0}; 
\end{scope}
  
\begin{scope}[scale=0.2, xshift=11.5 cm, yshift=0.3 cm]
\draw (0,0) rectangle (1,2); \draw (0,1) rectangle (2,2);
\node at (.5,1.45) {\footnotesize 2}; 
\node at (1.5,1.45) {\footnotesize 2}; 
\node at (.5,0.45) {\footnotesize 0}; 
\end{scope}
  
\begin{scope}[scale=0.2, xshift=14.3 cm, yshift=1.7 cm]
\draw (0,0) rectangle (1,2); \draw (0,1) rectangle (2,2);
\node at (.5,1.45) {\footnotesize 3};  
\node at (1.5,1.45) {\footnotesize 2}; 
\node at (.5,0.45) {\footnotesize 0}; 
\end{scope}
  
\begin{scope}[scale=0.2, xshift=17.3 cm, yshift=0.7 cm]
\draw (0,0) rectangle (1,2); \draw (0,1) rectangle (2,2);
\node at (.5,1.45) {\footnotesize 3}; 
\node at (1.5,1.45) {\footnotesize 3}; 
\node at (.5,0.45) {\footnotesize 0}; 
\end{scope}
  
\begin{scope}[scale=0.2, xshift=20 cm, yshift=5 cm]
\draw (0,0) rectangle (1,2); \draw (0,1) rectangle (2,2);
\node at (.5,1.45) {\footnotesize 3};  
\node at (1.5,1.45) {\footnotesize 3}; 
\node at (.5,0.45) {\footnotesize 1}; 
\end{scope}
  
\begin{scope}[scale=0.2, xshift=15 cm, yshift=5 cm]
\draw (0,0) rectangle (1,2); \draw (0,1) rectangle (2,2);
\node at (.5,1.45) {\footnotesize 3}; 
\node at (1.5,1.45) {\footnotesize 2}; 
\node at (.5,0.45) {\footnotesize 1}; 
\end{scope}
  
\begin{scope}[scale=0.2, xshift=11.7 cm, yshift=6 cm]
\draw (0,0) rectangle (1,2); \draw (0,1) rectangle (2,2);
\node at (.5,1.45) {\footnotesize 3}; 
\node at (1.45,1.45) {\footnotesize 1}; 
\node at (.5,0.45) {\footnotesize 1}; 
\end{scope}

\begin{scope}[scale=0.2, xshift=22.5 cm, yshift=9.5 cm]
\draw (0,0) rectangle (1,2); \draw (0,1) rectangle (2,2);
\node at (.5,1.45) {\footnotesize 3};  
\node at (1.5,1.45) {\footnotesize 3}; 
\node at (.5,0.45) {\footnotesize 2}; 
\end{scope}

\begin{scope}[scale=0.2, xshift=17.5 cm, yshift=9.5 cm]
\draw (0,0) rectangle (1,2); \draw (0,1) rectangle (2,2);
\node at (.5,1.45) {\footnotesize 3};  
\node at (1.5,1.45) {\footnotesize 2}; 
\node at (.5,0.45) {\footnotesize 2}; 
\end{scope}
  
\begin{scope}[scale=0.2, xshift=13 cm, yshift=9.5 cm]
\draw (0,0) rectangle (1,2); \draw (0,1) rectangle (2,2);
\node at (.5,1.45) {\footnotesize 3}; 
\node at (1.5,1.45) {\footnotesize 1}; 
\node at (.5,0.45) {\footnotesize 2}; 
\end{scope}

\begin{scope}[scale=0.2, xshift=9.2 cm, yshift=10.5 cm]
\draw (0,0) rectangle (1,2); \draw (0,1) rectangle (2,2);
\node at (.5,1.45) {\footnotesize 3}; 
\node at (1.5,1.45) {\footnotesize 0}; 
\node at (.5,0.45) {\footnotesize 2}; 
\end{scope}  
 
\begin{scope}[scale=0.2, xshift=6.5 cm, yshift=9 cm]
\draw (0,0) rectangle (1,2); \draw (0,1) rectangle (2,2);
\node at (.5,1.45) {\footnotesize 2}; 
\node at (1.5,1.45) {\footnotesize 0}; 
\node at (.5,0.45) {\footnotesize 2}; 
\end{scope}  
  
\begin{scope}[scale=0.2, xshift=10.5 cm, yshift=14 cm]
\draw (0,0) rectangle (1,2); \draw (0,1) rectangle (2,2);
\node at (.5,1.45) {\footnotesize 3}; 
\node at (1.5,1.45) {\footnotesize 0}; 
\node at (.5,0.45) {\footnotesize 3}; 
\end{scope} 
 
\begin{scope}[scale=0.2, xshift=15.5 cm, yshift=14 cm]
\draw (0,0) rectangle (1,2); \draw (0,1) rectangle (2,2);
\node at (.5,1.45) {\footnotesize 3};  
\node at (1.5,1.45) {\footnotesize 1}; 
\node at (.5,0.45) {\footnotesize 3}; 
\end{scope} 
    
\begin{scope}[scale=0.2, xshift=20 cm, yshift=14 cm]
\draw (0,0) rectangle (1,2); \draw (0,1) rectangle (2,2);
\node at (.5,1.45) {\footnotesize 3};  
\node at (1.5,1.45) {\footnotesize 2}; 
\node at (.5,0.45) {\footnotesize 3}; 
\end{scope}
  
\begin{scope}[scale=0.2, xshift=25 cm, yshift=14 cm]
\draw (0,0) rectangle (1,2); \draw (0,1) rectangle (2,2);
\node at (.5,1.45) {\footnotesize 3};  
\node at (1.5,1.45) {\footnotesize 3}; 
\node at (.5,0.45) {\footnotesize 3}; 
\end{scope}
    
\end{tikzpicture}
\caption{The hyperplane arrangement $\mbox{Cat}^3(A_2)$ together with the Shi tableaux 
that correspond to each dominant region.}
\label{fig:hypera}
\end{center}
\end{figure}
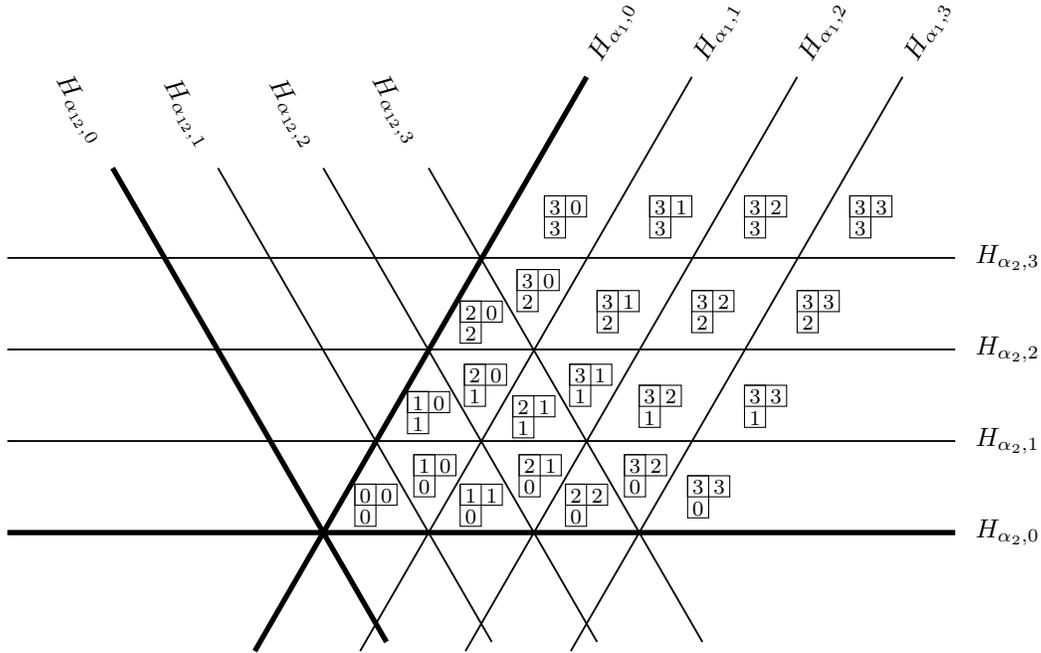

%!TEX root = type_A.tex
\section{From dominant regions $\rR^m_n$ to partitions in $\pP^m_n$}
\label{regions to facets}
In this section we prove Theorem \ref{all partitions} combinatorially.
In particular, we provide a bijection $\varphi_n$ between the set $\rR^m_n$ of dominant regions 
and the set $\pP^m_n$ of type-$A$ partitions. 
More precisely, for a simple root $\alpha$ of type $A_n$, the bijection $\varphi_n$ characterizes the dominant regions of $\rR^m_n$ 
which are separated from the origin by the hyperplane $H_{\alpha, m}$, 
it terms of partitions in $\pP^m_n$. 
The formula for $\varphi_n$ is given in Theorem~\ref{map1} and its inverse is explicitly described in 
Theorem~\ref{restatement}. 
Before we proceed to the proofs, 
a few comments on the history of this
bijection are in order. In \cite{st-haiot-96,st-hapfti-98}, Stanley
defined a bijection between regions in the extended
Shi arrangement and generalized parking functions. In both papers, the
bijection is recursively constructed. When restricted to the dominant
chamber, the parking functions in the image can be seen as partitions
and ours agrees with it. Richards \cite{ri-sdn-96} explictly defined a
bijection from dominant regions to partitions and again, ours agrees
with it. He proved it was an injection and used the sizes of the sets
involved to prove it was a bijection. In Theorem~\ref{restatement}, we
give a formula for the inverse of the function in \cite{ri-sdn-96}, thereby
showing directly that we have a bijection.

\

  Let $\tT^m_n$ denote the set of staircase tableaux of size $(n,n-1,\dots,1)$, whose entries are positive integers 
  between $0$ and $m$, 
  and let $k_{i,j}$ be the entry of the box $(i,n-j+1)$. 
  Note  that $\rR^m_n$ can be considered  as a proper 
  subset of $\tT^m_n$. We define $\phi_n:\tT^m_n\to \pP^m_n$ to be the map which sends each tableau 
  $T=\{k_{i,j} \mid 1\leq i\leq j\leq n\}\in\tT^m_n$ to the partition $\lambda_T=(\lambda_1,\lambda_2,\dots,\lambda_n)$ with parts 
  $\lambda_i=\sum \limits_{j=i}^n k_{i,j}$. Clearly $\phi_n$ is a surjection. Let $\varphi_n$ denote 
  the restriction of the map $\phi_n$ to the set $\rR^m_n$. 
  
\begin{theorem}
\label{map1}
  %Let $T=\{k_{i,j} \mid  1\leq i\leq j\leq n\}$ be the Shi tableau of a region $R\in\rR^m_n$. 
  %For every $i\in[n]$ define $\lambda_i=\sum \limits_{j=i}^n k_{i,j}$. 
  %The map $\varphi_n:\rR^m_n\to \pP^m_n$ which sends every Shi tableau $T$
  %to the partition $\lambda=(\lambda_1,\lambda_2,\dots,\lambda_n)$, 
  The map $\varphi_n$ is a bijection between the sets $\rR^m_n$ and $\pP^m_n$.
\end{theorem}

It is immediate from the definition of the Shi tableaux that the map $\varphi_n$ is well defined. 
In order to show that $\varphi_n$ is a bijection we will first construct its inverse. 
The first step towards this direction constitutes the next result. 

\begin{theorem}
\label{restatement} 
  Let $\lambda=(\lambda_1,\lambda_2,\dots,\lambda_n)\in\pP^m_n$. 
  For every $1\leq i\leq j\leq n$ define  
  \begin{equation}
  \label{kij}
    k_{i,j}(\lambda)=\min\left\{m,\left\lceil
    \frac{\lambda_i-
    \sum\limits_{\ell=j+1}\limits^n k_{i,\ell}(\lambda)+ 
    \sum\limits_{\ell= i+1}\limits^j k_{\ell,j}(\lambda)}{j-i+1} \right\rceil\right\}. 
  \end{equation} 
  The map $\bar{\varphi}_n:\pP^m_n\to \tT^m_n$, which sends each $\lambda$ to the set 
  $\{k_{i,j}(\lambda) \mid  1\leq i\leq j\leq n\}$, is an injection.   
\end{theorem} 

In the remainder of this section, when there is no ambiguity 
we will write $k_{i,j}$ instead of $k_{i,j}(\lambda)$. For the proof of Theorem  \ref{restatement}
we need the following lemma. 

\begin{lemma}
\label{decrease} 
  Let $\lambda\in\pP^m_n$, and let $k_{i,j}(\lambda)=k_{i,j}$ be defined as in Equation \eqref{kij}.  
  Then, $k_{i,j}\geq k_{i,j-1}$ for all $1\leq i< j\leq n$. 
\end{lemma}

\begin{proof}
  We proceed by induction on $n$. 
  Let $n=2$. We will show that $k_{1,2}\geq k_{1,1}$. If $k_{1,2}=m$ our claim is trivially true, 
  so we  may suppose that $k_{1,2}<m$.  
  Since $k_{1,1}=\min\{m,\lambda_1-k_{1,2}\}$, it is enough to show that $k_{1,2}\geq \lambda_1-k_{1,2}$, 
  or equivalently that $2k_{1,2}\geq\lambda_1$. 
  We have $2k_{1,2}= 2\left\lceil \frac{\lambda_1+k_{2,2}}{2}\right\rceil \geq \lambda_1+k_{2,2}.$ 
  Now, from  Equation \eqref{kij} we deduce that $\lambda_2=k_{2,2}$, and therefore 
  $2k_{1,2}\geq\lambda_1+\lambda_2\geq\lambda_1$ and we are done. 
 
  By applying induction, we may assume that $k_{l,n}\geq k_{l,n-1}$ for 
  every $2\leq \ell\leq n-1$, and thus it remains to show that $k_{1,n}\geq k_{1,n-1}$ as well.  
  As before, the result is trivial when $k_{1,n}=m$, so we may suppose that $k_{1,n}<m. $ 
  In this case 
  $k_{1,n}=\left\lceil\frac{\lambda_1+\sum_{\ell=2}^n k_{\ell,n}}{n}\right\rceil = 
  \frac{\lambda_1+\sum_{\ell=2}^n k_{\ell,n}+ \upsilon}{n}$ for some $\upsilon\in[n-1]_{\geq 0}$.  
  
  We first  show that 
  \[\frac{\lambda_1+\sum_{\ell=2}^n k_{\ell,n}}{n}\geq\frac{\lambda_1-k_{1,n}+\sum_{\ell=2}^{n-1} k_{\ell,n-1}}{n-1}.\]
  Indeed, the above inequality holds if and only if 
  \[\begin{aligned} 
  (n-1)(\lambda_1+\sum_{\ell=2}^n k_{\ell,n}) \geq n(\lambda_1-k_{1,n}+
  \sum_{\ell=2}^{n-1} k_{\ell,n-1}) &  \Leftrightarrow\\
  -\lambda_1+(n-1)\sum_{\ell=2}^n k_{\ell,n}+nk_{1,n}-n\sum_{\ell=2}^{n-1} k_{\ell,n-1} \geq 0 & \Leftrightarrow\\
  (n-1)\sum_{\ell=2}^n k_{\ell,n}+\sum_{\ell=2}^n k_{\ell,n}+\upsilon-n\sum_{\ell=2}^{n-1} k_{\ell,n-1} \geq 0 & \Leftrightarrow\\
  n\left(\sum_{\ell=2}^n k_{\ell,n}-\sum_{\ell=2}^{n-1} k_{\ell,n-1}\right)+\upsilon \geq 0 & \Leftrightarrow\\
        n \sum_{\ell=2}^{n-1} \left(  k_{\ell,n}- k_{\ell,n-1}\right)+  n k_{n,n}  + \upsilon \geq 0,
    \end{aligned} \] 
  which holds, since $k_{n,n}=\lambda_n \geq 0,\ \upsilon \geq 0$ and (by induction) $k_{l,n}\geq k_{l,n-1}$ for 
  every $2\leq \ell\leq n-1$. 
  
  Thus \[\frac{\lambda_1+\sum_{\ell=2}^n k_{\ell,n}}{n}\geq\frac{\lambda_1-k_{1,n}+\sum_{\ell=2}^{n-1} k_{\ell,n-1}}{n-1},\] 
  and therefore 
  \[\left\lceil\frac{\lambda_1+\sum_{\ell=2}^n k_{\ell,n}}{n}\right\rceil\geq
  \left\lceil\frac{\lambda_1-k_{1,n}+\sum_{\ell=2}^{n-1} k_{\ell,n-1}}{n-1}\right\rceil,\] 
  or equivalently, since $k_{1,n}<m$, 
  \[k_{1,n}\geq\left\lceil\frac{\lambda_1-k_{1,n}+\sum_{\ell=2}^{n-1} k_{\ell,n-1}}{n-1}\right\rceil.\]
  Finally, since
  $k_{1,n-1}=\min\left\{m,\left\lceil\frac{\lambda_1-k_{1,n}+\sum_{\ell=2}^{n-1} k_{\ell,n-1}}{n-1}\right\rceil\right\}$, 
  we deduce that  $k_{1,n}\geq k_{1,n-1}$. 
  This completes the induction and the proof of the lemma.
\end{proof}

\begin{proof}[Proof of Theorem \ref{restatement}]
  It suffices to show that 
  if $\phi_n(\bar{\varphi}_n(\lambda)):=\mu = (\mu_1,\mu_2,\dots,\mu_n)$, then $\mu=\lambda$. 
  We fix some $1\leq i \leq n$ and distinguish two cases: (1) $k_{i,i}(\lambda)<m$ and (2) $k_{i,i}(\lambda)=m$. 

  (1) Let $k_{i,i}(\lambda)<m$. Then, Equation (\ref{kij}) implies that   
  $k_{i,i}(\lambda)=\lambda_i-\sum\limits_{\ell=i+1}^n k_{i,\ell}(\lambda)$, therefore, 
  $\lambda_i=\sum\limits_{\ell=i}^n k_{i,\ell}(\lambda)=\mu_i$.

  (2) Let $k_{i,i}(\lambda)=m$. Then, it follows from Lemma \ref{decrease} that 
  $k_{i,\ell}(\lambda)=m$ for all $i<\ell\leq n$. Furthermore, since  $k_{i,i}(\lambda)=m$, Equation (\ref{kij}) implies  that 
  $\lambda_i-\sum\limits_{\ell=i+1}^n k_{i,\ell}(\lambda)\geq m$ or equivalently that  
  $\lambda_i \geq m+ \sum\limits_{\ell=i+1}^n k_{i,\ell}(\lambda)=m(n-i+1)$. 
  However, $\lambda\in \pP^m_n$, thus 
  $\lambda_i \leq m(n-i+1)$. 
  We therefore conclude that $\lambda_i=m(n-i+1)=\sum\limits_{\ell=i+1}^n k_{i,\ell}(\lambda)=\mu_i$. 

  Hence $(\mu_1,\mu_2,\dots,\mu_n)=\lambda$. From the above discussion we deduce that 
  for every $\lambda \in\pP^m_n$ we have  $\phi_n(\bar{\varphi}_n(\lambda))=\lambda$. 
  This implies that $\bar{\varphi}_n$ is an injection.  
\end{proof}

\begin{corollary}
\label{cor:kij}
Let $\lambda\in \pP^m_n$, and let $k_{i,j}(\lambda)=k_{i,j}$ be defined as in Equation (\ref{kij}).  
Then, $\lambda_i=\sum\limits_{j=i}\limits^nk_{i,j}$, for every $i\in[n]$.
\end{corollary}

\begin{proof}
The result follows immediately from Theorem \ref{restatement}.
\end{proof}

In view of Theorem \ref{restatement} and Corollary \ref{cor:kij}, in order 
to prove Theorem \ref{map1}, it suffices to show that $\bar{\varphi}_n(\pP^m)=\rR^m_n$ 
(see Lemma \ref{welldefined} below). For this proof we need the following two lemmas.

\begin{lemma}
\label{subtableau}
Let $\lambda\in \pP^m_n$  with $\bar{\varphi}_n(\lambda)=\{k_{i,j} \mid 1\leq i\leq j\leq n\}$.    
Consider the  partitions 
\begin{enumerate}
 \item $\lambda^{(1)}=(\lambda_2,\lambda_3,\dots,\lambda_n)$.   
 \item $\lambda^{(2)}$, with parts $\lambda^{(2)}_i=\lambda_i-k_{i,n}$, for every $1\leq i\leq n-1$. 
 \item $\lambda^{(3)}$, with parts $\lambda^{(3)}_i=\lambda^{(2)}_i-k_{i,n-1}$, for every $1\leq i\leq n-2$. 
 \item $\lambda^{(4)}$, with parts $\lambda^{(4)}_i=\lambda_i-k_{i,n-1}$, for every $1\leq i\leq n-1$.
%\item $\lambda^3$, with parts $\lambda^3_i=\sum\limits_{l=1}\limits^{n-i}k_{l(n+1-i)}$, for $1\leq i\leq n-1$, 
\end{enumerate}
Then, 
\begin{itemize} 
\item $\lambda^{(1)},\lambda^{(2)},\lambda^{(4)}\in\pP^m_{n-1}$ and 
\item $\lambda^{(3)}\in \pP^m_{n-2}$. 
\end{itemize}

Moreover, if
\[T^{(h)}=\{k^{(h)}_{i,j} \mid 1\leq i\leq j\leq n-1 \}:=\bar{\varphi}_{n-1}(\lambda^{(h)}), 
\mbox{ for } h\in\{1,2,4\}\]  and 
\[T^{(3)}=\{k^{(3)}_{i,j} \mid 1 \leq i \leq j \leq n-2 \}:=\bar{\varphi}_{n-2}(\lambda^{(3)}),\] 
then the following  relations hold:
\begin{equation}
\label{eq0}
k^{(1)}_{i,j}=k_{i+1,j+1},\ \mbox{for every}\ 1\leq i\leq j\leq n-1, 
\end{equation}
\begin{equation}
\label{eq1}
k^{(2)}_{i,j}=k_{i,j},\ \mbox{for every}\ 1\leq i\leq j\leq n-1, 
\end{equation}
\begin{equation}
\label{eq1b}
k^{(3)}_{i,j}=k_{i,j},\ \mbox{for every}\ 1\leq i\leq j\leq n-2,
\end{equation}
\begin{equation}
\label{eq1c}
k^{(3)}_{i,j}=k^{(4)}_{i,j},\ \mbox{for every}\ 1\leq i\leq j\leq n-2,
\end{equation}
and 
\begin{equation}
\label{eq1d}
k_{i,j}^{(4)}=\begin{cases} k_{i,j+1}, \ \ \mbox{if}\ j=n-1  \\
                     k_{i,j}, \ \ \ \ \ \mbox{if}\ j\leq n-2, 
	      \end{cases}  
\ \mbox{for every}\ 1\leq i\leq j\leq n-1.
\end{equation}
\end{lemma}
\medskip

\noindent Equations (\ref{eq0})-(\ref{eq1d}) explain the relation between each partition 
$\lambda$ with tableau $T$ and  partition $\lambda^{(h)}$ with tableau $T^{(h)}$, for $h=1,2,3,4$. 
In particular:
\begin{enumerate}
 \item  The partition $\lambda^{(1)}$ is obtained from $\lambda$ by omitting 
   its first part, 
   while the tableau $ T^{(1)}=\bar{\varphi}_{n-1}(\lambda^{(1)})$ 
   is obtained from $T$ by removing its top row. 
  \item  The partition $\lambda^{(2)}$ is obtained from $\lambda$ by subtracting  
   (for each $i$) the  $i$-th entry of the leftmost column of $T$ from  $\lambda_i$, 
   while the tableau $T^{(2)}=\bar{\varphi}_{n-1}(\lambda^{(2)})$ is 
   obtained from $T$ by removing its leftmost column. 
 \item  The partition $\lambda^{(3)}$ is obtained from $\lambda$ by subtracting (for each $i$) 
   the  $i$-th entries of the two leftmost columns of $T$ from $\lambda_i$, 
   while the tableau $T^{(3)}=\bar{\varphi}_{n-2}(\lambda^{(3)})$ is 
   obtained from $T$ by removing its  two leftmost columns.  
 \item The partition $\lambda^{(4)}$ is obtained from $ \lambda$ by subtracting (for each $i$) the
  $i$-th entry of the  second leftmost column of $T$ from $\lambda_i$, 
   while  the tableau $T^{(4)}=\bar{\varphi}_{n-1}(\lambda^{(4)})$ is 
   obtained from $T$ by removing its second leftmost column  as well as its
   bottom leftmost box.
\end{enumerate}

Figure \ref{22} illustrates the tableaux $T$ and $ T^{(h)}$ for $h\in{1,2,3,4}$.

\begin{proof}We first show that $\lambda^{(1)},\lambda^{(2)},\lambda^{(4)}\in\pP^m_{n-1}$ and 
$\lambda^{(3)}\in \pP^m_{n-2}$.
The fact that $\lambda^{(1)}\in \pP^m_{n-1}$ is immediate from its definition.  
For $\lambda^{(2)}$ recall that $k_{i,j}\in[m]_0$, for all $1\leq i\leq j\leq n$,  
and also that $\lambda_i=\sum\limits_{j=i}\limits^nk_{i,j}$ for every $i\in[n]$ (Corollary \ref{cor:kij}).
We therefore conclude that $ \lambda^{(2)}_i = \lambda_i-k_{i,n}\leq m(n-i)$, for every $i\in[n]$, 
which shows that $\lambda^{(2)}\in\pP^m_{n-1}$. 
Our claim for $\lambda^{(3)}$ and $\lambda^{(4)} $ is completely analogous.  

Equations (\ref{eq0})-(\ref{eq1c}) follow directly from the definition of $k_{i,j}^{(h)}$. 
We focus on Equation (\ref{eq1d}). We first prove the argument for $j=n-1$. 
Let  $1 \leq i \leq n-1$. Using Equations (\ref{eq1}), (\ref{eq1b}) and (\ref{eq1c})  we get  the following 
implications: 
 \[\begin{aligned} 
   \lambda_i & = k_{i,i}+k_{i,i+1}+\cdots+k_{i,n-1}+k_{i,n}\Rightarrow \\
   \lambda_i-k_{i,n-1} & = k_{i,i}+k_{i,i+1}+\cdots+k_{i,n-1}+k_{i,n}\Rightarrow \\
   \lambda^{(4)}_j & = k_{i,i}+k_{i,i+1}+\cdots+k_{i,n-1}+k_{i,n}\Rightarrow \\
   k^{(4)}_{i,i}+k^{(4)}_{i,i+1}+\cdots+k^{(4)}_{i,n-1} & = k_{i,i}+k_{i,i+1}+\cdots+k_{i,n-1}+k_{i,n}\Rightarrow \\
   k^{(3)}_{i,i}+k^{(3)}_{i,i+1}+\cdots+k^{(3)}_{i,n-2}+k^{(4)}_{i,n-1} & = k_{i,i}+k_{i,i+1}+\cdots+k_{i,n-1}+k_{i,n}\Rightarrow \\
   k_{i,i}+k_{i,i+1}+\cdots+k_{i,n-2}+k^{(4)}_{i,n-1} & = k_{i,i}+k_{i,i+1}+\cdots+k_{i,n-1}+k_{i,n} \Rightarrow \\
   k^{(4)}_{i,n-i} & = k_{i,n}. 
   \end{aligned} \] 
For $j\leq n-2$ the result follows directly from Equations (\ref{eq1b}) and (\ref{eq1c}). Indeed, for every $i\in[n-2]$ 
we have 
$k^{(4)}_{i,j}=k^{(3)}_{i,j}=k_{i,j}$.  
\end{proof}

\begin{figure}
\begin{center}
\begin{tikzpicture}[scale=0.6]
\small%\footnotesize
\draw (3,1) rectangle (7,2);
\node at (3.5,1.5){$k_{1,4}$};\node at (4.5,1.5){$k_{1,3}$};\node at (5.5,1.5){$k_{1,2}$};\node at (6.5,1.5){$k_{1,1}$}; 
\draw (3,0) rectangle (6,2);
\node at (3.5,0.5){$k_{2,4}$};\node at (4.5,0.5){$k_{2,3}$};\node at (5.5,0.5){$k_{2,2}$};
\draw (3,-1) rectangle (5,2);
\node at (3.5,-0.5){$k_{3,4}$};\node at (4.5,-0.5){$k_{3,3}$};
\draw (3,-2) rectangle (4,2);
\node at (3.5,-1.5){$k_{4,4}$};
\begin{scope}[xshift=5cm, yshift=0cm]
\draw (3,1) rectangle (7,2);
\draw (3,0) rectangle (6,2);
\node at (3.5,0.5){$k_{2,4}$};\node at (4.5,0.5){$k_{2,3}$};\node at (5.5,0.5){$k_{2,2}$};
\draw (3,-1) rectangle (5,2);
\node at (3.5,-0.5){$k_{3,4}$};\node at (4.5,-0.5){$k_{3,3}$};
\draw (3,-2) rectangle (4,2);
\node at (3.5,-1.5){$k_{4,4}$};
\end{scope}
\begin{scope}[xshift=10cm, yshift=0cm]
\draw (3,1) rectangle (7,2);
\node at (4.5,1.5){$k_{1,3}$};\node at (5.5,1.5){$k_{1,2}$};\node at (6.5,1.5){$k_{1,1}$}; 
\draw (3,0) rectangle (6,2);
\node at (4.5,0.5) {$k_{2,3}$};\node at (5.5,0.5) {$k_{2,2}$};
\draw (3,-1) rectangle (5,2);
\node at (4.5,-0.5) {$k_{3,3}$};
\draw (3,-2) rectangle (4,2);
\end{scope}
\begin{scope}[xshift=15cm, yshift=0cm]
\draw (3,1) rectangle (7,2);
\node at (5.5,1.5){$k_{1,2}$};\node at (6.5,1.5){$k_{1,1}$}; 
\draw (3,0) rectangle (6,2);
\node at (5.5,0.5){$k_{2,2}$};
\draw (3,-1) rectangle (5,2);
\draw (3,-2) rectangle (4,2);
\end{scope}
\begin{scope}[xshift=20cm, yshift=0cm]
\draw (3,1) rectangle (7,2);
\node at (3.5,1.5) {$k_{1,4}$}; 
\node at (5.5,1.5) {$k_{1,2}$};\node at (6.5,1.5){ $k_{1,1}$}; 
\draw (3,0) rectangle (6,2);
\node at (3.5,0.5) {$k_{2,4}$};
\node at (5.5,0.5) {$k_{2,2}$};
\draw (3,-1) rectangle (5,2);
\node at (3.5,-0.5) {$k_{3,4}$};
\draw (3,-2) rectangle (4,2);
\end{scope}

\begin{scope}[yshift=-5.5cm]\small%\footnotesize
\draw (3,1) rectangle (7,2);
\node at (3.5,1.5){$k_{1,4}$};\node at (4.5,1.5){$k_{1,3}$};\node at (5.5,1.5){$k_{1,2}$};\node at (6.5,1.5){$k_{1,1}$}; 
\draw (3,0) rectangle (6,2);
\node at (3.5,0.5){$k_{2,4}$};\node at (4.5,0.5){$k_{2,3}$};\node at (5.5,0.5){$k_{2,2}$};
\draw (3,-1) rectangle (5,2);
\node at (3.5,-0.5){$k_{3,4}$};\node at (4.5,-0.5){$k_{3,3}$};
\draw (3,-2) rectangle (4,2);
\node at (3.5,-1.5){$k_{4,4}$};
\begin{scope}[xshift=5cm, yshift=0cm]
\draw (4,1) rectangle (7,2);
\node at (4.5,1.5){$k^{(1)}_{1,3}$};\node at (5.5,1.5){$k^{(1)}_{1,2}$};\node at (6.5,1.5){$k^{(1)}_{1,1}$}; 
\draw (4,0) rectangle (6,2);
\node at (4.5,0.5){$k^{(1)}_{2,3}$};\node at (5.5,0.5){$k^{(1)}_{2,2}$};
\draw (4,-1) rectangle (5,2);
\node at (4.5,-0.5){$k^{(1)}_{3,3}$};
\end{scope}
\begin{scope}[xshift=10cm, yshift=0cm]
\draw (4,1) rectangle (7,2);
\node at (4.5,1.5){$k^{(2)}_{1,3}$};\node at (5.5,1.5){$k^{(2)}_{1,2}$};\node at (6.5,1.5){$k^{(2)}_{1,1}$}; 
\draw (4,0) rectangle (6,2);
\node at (4.5,0.5){$k^{(2)}_{2,3}$};\node at (5.5,0.5){$k^{(2)}_{2,2}$};
\draw (4,-1) rectangle (5,2);
\node at (4.5,-0.5){$k^{(2)}_{3,3}$};
\end{scope}
\begin{scope}[xshift=15cm, yshift=0cm]
\draw (5,1) rectangle (7,2);
\node at (5.5,1.5) {$k^{(3)}_{1,2}$};\node at (6.5,1.5){$k^{(3)}_{1,1}$}; 
\draw (5,0) rectangle (6,2);
\node at (5.5,0.5) {$k^{(3)}_{2,2}$};
\end{scope}
\begin{scope}[xshift=20cm, yshift=0cm]
\draw (4,1) rectangle (7,2);
\node at (4.5,1.5){$k^{(4)}_{1,3}$};\node at (5.5,1.5){$k^{(4)}_{1,2}$};\node at (6.5,1.5){$k^{(4)}_{1,1}$}; 
\draw (4,0) rectangle (6,2);
\node at (4.5,0.5){$k^{(4)}_{2,3}$};\node at (5.5,0.5){$k^{(4)}_{2,2}$};
\draw (4,-1) rectangle (5,2);
\node at (4.5,-0.5){$k^{(4)}_{3,3}$};
\end{scope}
\end{scope}
\end{tikzpicture}
\caption{The tableau $T$ is displayed on the left 
followed by the tableaux of the partitions $\lambda^{(h)}$, for $h=1,2,3,4$, respectively.}  
\label{22}
\end{center}
\end{figure}

Lemma \ref{decrease} states that the entries $k_{i,j}$ from Equation (\ref{kij}) decrease from left to right along the rows. 
The following lemma shows that an analogous property holds for the columns. More precisely, we show that 
the entries $k_{i,j}$ decrease from top to bottom.
\begin{lemma}
\label{decrease2}
  Let $\lambda\in\pP^m_n$ and $k_{i,j}(\lambda)=k_{i,j}$ defined as in Equation (\ref{kij}).  
  Then, $k_{i,j}\geq k_{i+1,j}$ for all $1\leq i< j\leq n$. 
\end{lemma}

\begin{proof}
  We proceed again by induction on $n$. 
  Let $n=2$. We will show that $k_{1,2}\geq k_{2,2}$. 
  We may assume that $k_{1,2}<m$ since otherwise the result is trivial. 
  Then, $k_{1,2}=\left\lceil\frac{\lambda_1+k_{2,2}}{2}\right\rceil\geq\left\lceil\frac{k_{1,1}+k_{1,1}}{2}\right\rceil=k_{1,1}$. 
  By induction and using the tableaux $T^{(h)}$ of Lemma \ref{subtableau} the result is trivial. 
\end{proof}

\begin{lemma}
\label{welldefined}
Let $\bar{\varphi}_n$ be the map of Theorem \ref{restatement}. Then $\bar{\varphi}_n(\pP^m_n)=\rR^m_n$.
\end{lemma}

\begin{proof}
We will first prove that $\bar{\varphi}_n(\pP^m_n)\subseteq \rR^m_n$. 
Let $\lambda \in \pP^m_n$. We will show that 
$\bar{\varphi}_n(\lambda)=\{ k_{i,j}(\lambda),\ 1 \leq i \leq j \leq n \}$ is an element of $\rR^m_n$.  
Equivalently,  we will show that for every $1 \leq i \leq j \leq n$, the numbers  $k_{i,j}$ 
satisfy the  Shi conditions.  

We proceed by induction on $n$. The result is trivial for $n=1$. We treat now the case $n=2$.   
Suppose first that $k_{1,2}<m$. Then $k_{1,2}=\frac{\lambda_1+k_{2,2}+\upsilon}{2}$, 
where $\upsilon\in\{0,1\}$. 
Since $\lambda_1=k_{1,2}+k_{2,2}$, we conclude that $k_{1,2}=k_{1,1}+k_{2,2}+\upsilon$ and we are done. 
Suppose now that $k_{1,2}=m$. Then $m\leq \frac{\lambda_1+k_{2,2}+\upsilon}{2}$, where $\upsilon\in\{0,1\}$. 
As before we conclude that $k_{1,1}+k_{2,2}\geq m-\upsilon\geq m-1$. Therefore the Shi condition hold for this case as well. 

We assume that  our claim holds for all $\ell <n$ and we will show that it holds for $n$ as well.
Let $\lambda=(\lambda_1,\ldots, \lambda_n)\in\pP^m_n $ and 
consider the partitions $\lambda^{(1)},\lambda^{(2)},\lambda^{(3)}$ and $\lambda^{(4)}$  
defined in Lemma \ref{subtableau}. 
Let $T^{(h)}=\{k^{(h)}_{i,j} \mid 1\leq i\leq j\leq n-1 \}:=\bar{\varphi}_{n-1}(\lambda^{(h)})$ 
for $h\in\{1,2,4\}$ and 
$T^{(3)}= \{ k^{(3)}_{i,j} \mid  1 \leq i \leq j \leq n-2 \}:=\bar{\varphi}_{n-2}(\lambda^{(3)})$.   
From the induction hypothesis, we have that $T^{(h)}\in\rR^m_{n-1}$ for all $h\in\{1,2,4\}$ 
and $T^{(3)} \in \rR^m_{n-2}$. 
Thus, from Equations (\ref{eq1}), (\ref{eq1b}), (\ref{eq1d}) and  the induction  hypothesis  we have that  
that the Shi condition (\ref{Ashi}) holds for each triplet 
$\{k_{i,j}, k_{i,l}, k_{l+1,j}\}$ with $(i,l,j)\neq(1,n-1,n)$. 

Hence, it remains to  show 
that the Shi condition (\ref{Ashi}) holds for the triplet $\{k_{1,n},k_{1,n-1},k_{n,n}\}$ as well. 
We distinguish two cases: (1) $k_{1,n}<m$ and (2) $k_{1,n}=m$. 

(1) Let $k_{1,n}<m$. Lemma \ref{decrease}~(i) implies that $k_{1,n-1}<m$. 
In this case condition (\ref{Ashi}) reduces to showing that  $k_{1,n}-k_{n,n}-k_{1,n-1}\in\{0,1\}$. 
By the  definition of $k_{1,n}$, we have that $k_{1,n}=\frac{\lambda_1+\sum\limits_{l=2}\limits^nk_{l,n}+\upsilon}{n}$, 
where $\upsilon\in[n-1]_{\geq 0}$. 
Moreover, $\lambda_1= \sum\limits_{\ell=1}\limits^nk_{1,\ell}$. We  therefore have: 
\begin{align} 
 nk_{1,n}=\lambda_1+\sum\limits_{l=2}\limits^nk_{l,n}+\upsilon& =
\sum\limits_{\ell=1}\limits^nk_{1,\ell}+\sum\limits_{l=2}\limits^nk_{l,n}+\upsilon=
k_{1,n}+\sum\limits_{\ell=1}\limits^{n-1}k_{1,\ell}+\sum\limits_{\ell=2}\limits^nk_{\ell,n}+\upsilon \leftrightarrow \notag\\
(n-1)k_{1,n}&=\sum\limits_{\ell=1}\limits^{n-1}k_{1,\ell}+\sum\limits_{\ell=2}\limits^nk_{\ell,n}+\upsilon.
 \label{eq3}
\end{align}
   
We recall now that $k_{1,n-1}<m$, thus for every $1\leq \ell\leq n-2$ we have that 
$k_{1,n-1}=k_{1,\ell}+k_{\ell+1,n-1}+\delta_{\ell}$, where $\delta_{\ell}\in\{0,1\}$. 
Hence,
\begin{align} 
\sum\limits_{\ell=1}\limits^{n-2}k_{1,n-1} & = 
\sum\limits_{\ell=1}\limits^{n-2}k_{1,\ell}+
\sum\limits_{\ell=1}\limits^{n-2}k_{\ell,n-1}+
\sum\limits_{\ell=1}\limits^{n-2}\delta_{\ell}\Leftrightarrow\\
\sum\limits_{\ell=1}\limits^{n-2}k_{1,\ell} & = 
\sum\limits_{\ell=1}\limits^{n-2}k_{1,n-1}- 
\sum\limits_{\ell=1}\limits^{n-2}k_{\ell,n-1}-
\sum\limits_{\ell=1}\limits^{n-2}\delta_{\ell}\Leftrightarrow
  \notag \\
\sum\limits_{\ell=1}\limits^{n-2}k_{1,\ell} & = 
(n-2)k_{1,n-1}- 
\sum\limits_{\ell=1}\limits^{n-2}k_{\ell,n-1}-
\sum\limits_{\ell=1}\limits^{n-2}\delta_{\ell}.
\label{eq4} 
\end{align}

On the other hand, for every $2\leq \ell \leq n-1$, the entry 
$k_{\ell,n}$ satisfies the Shi condition, therefore, 
$k_{\ell,n}=k_{\ell,n-1}+k_{n,n}+\delta'_{\ell}$, where $\delta'_{\ell}\in\{0,1\}$. 
Thus,
  
\begin{align} 
\sum\limits_{\ell=2}\limits^{n-1}k_{\ell,n} & = 
\sum\limits_{\ell=2}\limits^{n-1}k_{\ell,n-1}+
\sum\limits_{\ell=2}\limits^{n-1}k_{n,n}+
\sum\limits_{\ell=2}\limits^{n-1}\delta'_{\ell}\Leftrightarrow \notag   \\  
\sum\limits_{\ell=2}\limits^{n-1}k_{\ell,n} & = 
\sum\limits_{\ell=2}\limits^{n-1}k_{\ell,n-1}+
(n-2)k_{n,n}+
\sum\limits_{\ell=2}\limits^{n-1}\delta'_{\ell}.   
\label{eq5}  
\end{align}
  % \end{equation}  

Substituting  Equations (\ref{eq4}) and (\ref{eq5}) in  Equation (\ref{eq3}) we get: 
\begin{align}
 (n-1)k_{1,n}&=(n-1)(k_{1,n-1}+k_{n,n})- 
\sum\limits_{\ell=1}\limits^{n-2}k_{\ell,n-1}-
\sum\limits_{\ell=1}\limits^{n-2}\delta_{\ell}+
\sum\limits_{\ell=2}\limits^{n-1}k_{\ell,n-1}+
\sum\limits_{\ell=2}\limits^{n-1}\delta'_{\ell}+\upsilon. 
\notag
\end{align} 
Hence,
\[k_{1,n}=k_{1,n-1}+k_{n,n}+\kappa,\]
where $\kappa=\frac{\sum\limits_{\ell=2}\limits^{n-1}\delta'_{\ell}-\sum\limits_{\ell=1}\limits^{n-2}\delta_{\ell}+\upsilon}{n-1}$.
Since $k_{1,n}, k_{1,2},k_{2,n} \in \mathbb{N}$, it follows that $\kappa \in \mathbb{Z}$.  
Moreover, from the fact that  $\delta_{\ell},\delta'_{\ell}\in \{0,1\}$ and $\upsilon\in[n-1]_{\geq 0}$, one can easily see  
that $\kappa\in\{0,1\}$, which completes our claim for $k_{1,n}< m$.  

(2) Let $k_{1,n}=m$. In this case,  showing that the triplet $\{k_{1,n},k_{1,n-1},k_{n,n}\}$  
satisfies the Shi condition, reduces to showing that 
$k_{1,n-1}+k_{n,n} \geq m-1$. 
Suppose that $k_{1,\ell}=m$ for some $1\leq\ell\leq n-1$. 
Then, Lemma \ref{decrease} implies that $k_{1,n-1}=m$ as well, thus the result follows. 
We argue similarly if $k_{\ell,n}=m$ for some $2\leq \ell\leq n$ (in this case we apply Lemma \ref{decrease2}). 
Finally, suppose that $k_{1,j},k_{i,n}<m$ for every $1\leq i\leq n-1$ and $2\leq j\leq n$.
Then, the Shi conditions for each triplet are always as in the upper case of (\ref{Ashi}). 
We can therefore follow verbatim the computation done in the case where $k_{1,n}<m$.   
In particular, the  definition of $k_{1,n}$ implies that  
$k_{1,n}=m\leq \left\lceil\frac{\lambda_1+\sum\limits_{l=2}\limits^nk_{\ell,n}}{n}\right\rceil$
which, after the computations, gives $m=k_{1,n}\leq k_{1,n-1}+k_{n,n}+\kappa$, where $\kappa\in\{0,1\}$. 
Therefore, $\bar{\varphi}_n(\pP^m_n)\subseteq\rR^m_n$.

We will prove now that $\rR^m_n\subseteq \bar{\varphi}_n(\pP^m_n)$. 
Let $T=\{k_{i,j} \mid  1 \leq i\leq j\leq n\}\in \rR^m_n$ and consider the 
partition $\lambda$ whose parts are the sums of the entries of each row of $T$. 
Clearly then $\lambda\in \pP^m_n$, and from the proof of Theorem 
\ref{restatement} it follows that  $T=\bar{\varphi}_n(\lambda)$, which implies that 
$T\in \bar{\varphi}_n(\pP^m_n)$. 
This completes the proof of the lemma.  
%Since $\phi_n$ is an injection, it follows that $\pP^m_n\subseteq\phi_n^{-1}(\rR^m(A_n))$. 
%Furthermore, from the proof  of Theorem \ref{restatement} we have that for every tableau with entries 
%$\{k_{i,j}: 1 \leq i  \leq j \leq n\}$, the inverse $\psi^{-1}_n ( \{k_{i,j}: 1 \leq i  \leq j \leq n\})$ is 
%the partition $(\lambda_1,\ldots, \lambda_n)$, 
%where $ \lambda_{n-j+1}  =\sum_{\ell=1}^j k_{\ell,j}$. 
%Therefore, $\psi^{-1}_n ( \{k_{i,j}: 1 \leq i  \leq j \leq n\})\in \pP^m_n$, 
%and $\psi^{-1}_n(\rR^m(A_n))\subseteq S^m(n)$. Thus $\rR^m(A_n)\subseteq \phi_n(S^m(n))$, 
%which implies that    
\end{proof}

\begin{proof}[Proof of Theorem \ref{map1}]
 The result follows immediately from Theorem \ref{restatement}, Corollary \ref{cor:kij} and Lemma \ref{welldefined}.
\end{proof}

\begin{proof}[Proof of Theorem \ref{all partitions}] 
From Theorem \ref{map1} we have that 
$\varphi_n:\rR^m_n\rightarrow \pP^m_n$ is a bijection. 
Finally, to show that $\varphi_n$ has the property in the statement of the theorem, 
notice that $\lambda_i=m(n-i+1)$ if and only if  
$k_{i,i}=k_{i,i+1}=\cdots=k_{i,n}=m$, which in turn implies that $H_{\alpha_i,m}$ is a separating wall. 
\end{proof}

%!TEX root = type_A.tex
\section{From facets of $\dD^m_n$ to partitions in $\pP^m_n$}
\label{facets to partitions}
In this section we prove Theorem \ref{snakedisA} combinatorially.
In particular, we give a bijection that characterizes the facets 
of $\dD^m_n$ with respect to the negative simple roots they contain, 
in terms of partitions in $P^m_n$. 

\begin{definition}
Let $P$ be a polygon with $m(n+1)+2$ vertices labeled by distinct integers 
of the set $[m(n+1)+1]_{\geq 0}$ and let $d$ be an $m$-diagonal of $P$. 
The smallest label between the endpoints of $d$ is called \emph{initial} point of $d$. 
\end{definition}

Consider now an $m(n+1)+2$-gon $P$ with its vertices labeled by the integers 
in $[m(n+1)+1]_{\geq 0}$ in counterclockwise order. Let also $D$ be a maximal $m$-dissection 
of $P$ and let 
$\{t_1,t_2,\ldots,t_n\}$ be the multiset consisting of the initial points of the diagonals in
$D$. Without loss of generality, we may assume that $t_1\geq t_2\geq \cdots
\geq t_n$. One can show that $t_i \leq m(n-i+1)$ for all $1 \leq i \leq n$, therefore, 
$(t_1,t_2,\ldots,t_n)$ is an element of $\pP^m_n$. 
We consider the map $\psi_n':\dD^m_n\to \pP^m_n$, which sends every dissection 
$D\in\dD^m_n$ to the partition defined by the initial points of $D$, and
note that it is a bijection. 
Even though the map $\psi_n'$ describes an obvious way to associate
integer partitions to polygon dissections, it can not be adopted in
our case, since the composition $\varphi_n^{-1}\circ \psi_n'$ does not
preserve the property of Theorem~\ref{propA}. Indeed, let
$D$ be the dissection of an $(m(n+1)+2)$-gon with diagonals
$\{im,mn+1\}$ where $1\leq i\leq n$. These diagonals, which are all
incident to $mn+1,$ do not form a ``snake'' when $ n>2$. Thus, the
dissection $D$ does not contain all the negative simple roots. Now,
applying the above bijection we have $\psi_n'(D)=(mn,m(n-1),\dots,m)$. On
the other hand, from Theorem~\ref{all partitions}, the partition
$(mn,m(n-1),\dots,m)$ corresponds to the Shi tableau with all entries
equal to $m$. This tableau represents the region having separating walls all
hyperplanes of the form $H_{{\alpha},m}$, where $\alpha$ is a simple root,
which implies that the bijection $\varphi_n^{-1}\circ \psi_n'$ does not
preserve the property stated in Theorem~\ref{propA}. 

Although the map $\psi_n'$ does not give us the required characterization,
it constitutes the key-idea for the Proof of Theorem~\ref{snakedisA}.
Based on the bijection $\psi_n'$, we relabel the vertices of $P$ so that
the property of Theorem~\ref{propA} is preserved. The
relabeling will give us an equivalent way to describe the negative simple
roots, which will be consistent with the property we want to preserve.
In the remainder of the paper, when we write that a vertex of a
polygon $P$ lies \emph{on the right} (resp. \emph{on the left}) of
some other vertex of $P$, we mean on the right (resp. on the left) 
with respect to the center of the polygon.

\begin{definition} 
\label{alternA}
Let $P$ be a polygon with $m(n+1)+2$ vertices labeled from $0$ to
$m(n+1)+1$ as follows: Fix a vertex labeled with 0.
The vertices on the right of $0$ are labeled with the
numbers $ k \in[m(m+1)+1]$ for which $ \lfloor
\frac{k}{m} \rfloor$ is even and so that the labels increase in
the counterclockwise direction from vertex 0. Similarly, the
vertices on the left of $0$ are labeled with those 
$k \in[m(m+1)+1]$ for which $ \lfloor \frac{k}{m} \rfloor$
is odd and so that the labels increase in the clockwise
direction.  We call this the \emph{alternating type-$A$ labeling.}
\end{definition}
 
\subsubsection{Representation of negative simple roots for type $A$} 
Let $P$ be an $(m(n+1)+2)$-gon, fix some vertex $0$ and consider its  alternating type $A$-labeling. 
For each $1\leq i\leq n$ we identify the negative simple root $-\alpha_i$ 
with the diagonal having endpoints $(n-i+1)m$ and $(n-i+2)m$. 
Notice that again the negative simple roots form a ``snake'', and thus 
the  colored positive roots can be defined as in Section \ref{subsec:dis}. 
For instance, for  $m=3$ and $n=4$, the  negative simple roots  $-\alpha_1,-\alpha_2,-\alpha_3,-\alpha_4$ correspond 
to the diagonals $\{12,15\},\{9,12\},\{6,9\},\{3,6\}$ respectively (see Figure \ref{fig:Snake-bij}).

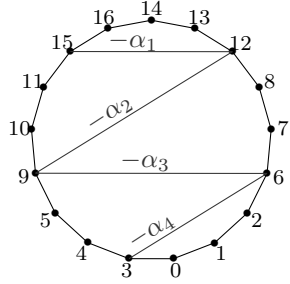
\begin{figure}
\begin{center}
\begin{tikzpicture}[scale=0.8]
\path (280.59:2 cm) coordinate (A0);\path (301.76:2 cm) coordinate (A1); 
\path (322.93:2 cm) coordinate (A2);\path (344.1:2 cm) coordinate (A3);
\path (365.27:2 cm) coordinate (A4);\path (26.44:2 cm) coordinate (A5); 
\path (47.61:2 cm) coordinate (A6);\path (68.78:2 cm) coordinate (A7);
\path (89.95:2 cm) coordinate (A8);\path (111.12:2 cm) coordinate (A9);
\path (132.29:2 cm) coordinate (A10);\path (153.46:2 cm) coordinate (A11); 
\path (174.63:2 cm) coordinate (A12);\path (195.8:2 cm) coordinate (A13);
\path (216.97:2 cm) coordinate (A14);\path (238.14:2 cm) coordinate (A15);
\path (259.31:2 cm) coordinate (A16);

\draw (A0)--(A1)--(A2)--(A3)--(A4)--(A5)--(A6)--(A7)--(A8)--(A9)--(A10)--(A11)--(A12)--(A13)--(A14)--(A15)--(A16)--(A0);

\node at (A0){\tiny $\bullet$};\node at (A1){\tiny $\bullet$};\node at (A2){\tiny$\bullet$};  
\node at (A3){\tiny $\bullet$};\node at (A4){\tiny $\bullet$};\node at (A5){\tiny $\bullet$};  
\node at (A6){\tiny $\bullet$};\node at (A7){\tiny $\bullet$};\node at (A8){\tiny $\bullet$}; 
\node at (A9){\tiny $\bullet$};\node at (A10){\tiny $\bullet$};\node at (A11){\tiny $\bullet$};
\node at (A12){\tiny $\bullet$};\node at (A13){\tiny $\bullet$};\node at (A14){\tiny $\bullet$}; 
\node at (A15){\tiny$\bullet$};\node at (A16){\tiny $\bullet$};  

\node at (280.59:2.2 cm){\footnotesize 0};\node at (301.76:2.2 cm){\footnotesize 1};\node at (322.93:2.2 cm){\footnotesize 2};  
\node at (344.1:2.2 cm){\footnotesize 6};\node at (365.27:2.2 cm){\footnotesize 7};\node at (26.44:2.2 cm){\footnotesize 8}; 
\node at (47.61:2.2 cm){\footnotesize 12};\node at (68.78:2.2 cm){\footnotesize 13};\node at (89.95:2.2 cm){\footnotesize 14}; 
\node at (111.12:2.2 cm){\footnotesize 16};\node at (132.29:2.2 cm){\footnotesize 15};\node at (153.46:2.2 cm){\footnotesize 11}; 
\node at (174.63:2.2 cm){\footnotesize 10};\node at (195.8:2.2 cm){\footnotesize 9};\node at (216.97:2.2 cm){\footnotesize 5}; 
\node at (238.14:2.2 cm){\footnotesize 4};\node at (259.31:2.2 cm){\footnotesize 3};   

\draw[black!80] (A3)--(A16); \draw[black!80] (A3)--(A13);\draw[black!80] (A6)--(A13); \draw[black!80] (A6)--(A10);

\node at (-0,-1.5){\rotatebox{33}{\tc{black!80}{$-\alpha_4$}}}; 
\node at (-0.1,-0.35){\tc{black!80}{$-\alpha_3$}}; 
\node at (-0.7,0.45){\rotatebox{33}{\tc{black!80}{$-\alpha_2$}}}; 
\node at (-0.3,1.65){\tc{black!80}{$-\alpha_1$}};
\end{tikzpicture}  
\caption{The alternating  type-$A$ labeling together with the  negative simple roots, for $m=3$ and $n=4$.} 
\label{fig:Snake-bij}
\end{center}
\end{figure}

\begin{proof}[Proof of Theorem~\ref{snakedisA}]
Consider  an $(m(n+1)+2)$-gon $P$ with alternating type-$A$ labeling. 
Let $D\in \dD^m_n$ be a dissection 
with diagonals $d_1,d_2,\dots, d_n$.  
Let also $\{t_1,t_2,\dots, t_n\}$ be the multiset consisting of all the initial points $t_i$. 
Without loss of generality we may assume that $t_1\geq t_2 \geq \cdots \geq t_n$ and proceed as we did with $\psi_n'$ 
(described in the beginning of this subsection). 
That is, we set $\psi_n(D)=(t_1,t_2,\ldots,t_n)$. 

We first need to prove that $\psi_n$ is well defined,
or equivalently that $t_i\leq m(n-i+1)$ for all $1\leq i\leq n$. 
Assume the contrary and let $i_0\in[n]$ be the greatest index for which $t_{i_0}>m(n-i_0+1)$. 
Thus $t_i\leq m(n-i+1)$ for all $i_0< i\leq n$.   
Since $m(n-i_0+1)<t_{i_0} \leq t_{i_0+1}\leq\cdots\leq t_n$, we deduce that 
the $m$-diagonals $d_1,d_2,\ldots, d_{i_0}$ should lie in the $(mi_0+2)$-gon 
defined by the diagonal corresponding to the root $-\alpha_{i_0}$ and the vertices 
$m(n-i_0+1),\ldots, m(n-i_0+2)$. But this is a contradiction, since an
$(mi_0+2)$-gon cannot contain $i_0$ many $m$-diagonals. 
Thus the map $\psi_n$ is well defined. 

To see that $\psi_n$ is a bijection we construct its inverse.
We proceed by induction on $n$, the case $n=1$ being trivial. 
Assume that we have constructed the bijection for $n-1$. 
Let $(\lambda_1,\lambda_2,\ldots,\lambda_n)$ be an element of $\pP^m_n$ and 
$P$ be an $(m(n+1)+2)$-gon with the alternating type-$A$ labeling. 
We will construct a dissection $D\in \dD^m_n$ of $P$ with set of initial points 
$\{\lambda_1,\lambda_2,\dots,\lambda_n\}$. We consider the vertex of $P$ 
which is labeled by $\lambda_1$. 
Among the two vertices which lie $(m+1)$-many vertices apart from $\lambda_1$, 
we denote by $\bar{\lambda}_1$ be the one with the greater label.  
Since $\lambda_1\leq mn$, it follows from the definition of the alternating labeling 
that $\bar{\lambda}_1\in\{mn+1,mn+2,\ldots,m(n+1)+1\}$, which implies that $\lambda_1<\bar{\lambda}_1$. 
We set $d_1$ to be the diagonal of $P$ with endpoints $\lambda_1$ and $\bar{\lambda}_1$. 
Clearly, $d_1$ dissects $P$ into a $(m+2)$-gon and a $(mn+2)$-gon, which we denote by $P_1$.  
Note that $P_1$ contains all the vertices $\lambda_i$, for $2\leq i\leq n$ 
and possibly some with greater labels. 
From the induction hypothesis, we can associate to the partition 
$(\lambda_2,\lambda_3,\dots,\lambda_n)$ a dissection of the polygon $P_1$ having diagonals 
$d_1,d_2,\dots, d_{n-1}$, where $d_i=\{\lambda_i<\bar{\lambda}_i\}$.  
We leave it to the reader to check that the map which sends $\lambda$ to 
the dissection containing the diagonals $d_1,d_2,\dots,d_n$ is indeed the inverse 
of $\psi_n$. By the representation of the negative simple roots as a ``snake'', 
it follows that a dissection contains the negative simple root $-\alpha_i$ 
if and only if it contains the diagonal $\{m(n-i+1),m(n-i+2)\}$. 
By induction on $n$ one can show that this occurs if and only if the partition 
$\psi_n(D)$ has $i$-th part equal to $(n-i+1)m$. 
\end{proof}

%!TEX root = type_A.tex

\section{Conclusion and ongoing work}
\label{sec:concl}
We complete this paper with the proof of Theorem \ref{propA}, which is a direct consequence of Theorem \ref{all partitions} and 
Theorem \ref{snakedisA}. We also present an application of the maps $\varphi_n$ and $\psi_n$ of Theorems \ref{all partitions} 
and \ref{snakedisA}. Finally, we briefly discuss  on our ongoing work and state an open problem. 

\begin{proof}[Theorem \ref{propA}]
We consider the map $\omega_n:\dD^m_n\to \rR^m_n$ with $\omega_n=\varphi^{-1}_n\circ\psi_n$. 
The result follows directly from Theorems 
\ref{all partitions} and \ref{snakedisA} 
\end{proof}

\subsection{Application}
Using the bijections provided in Theorems \ref{all partitions} and \ref{snakedisA}, 
we give a combinatorial proof of the 
fact that the number of facets of $\Delta^m_+(A_n)$ as well as the number of regions in ${\rm Cat}^m(A_n)$ 
is equal to $\frac{1}{n+1}\binom{m(n+1)+n-1}{n}$.
\begin{corollary}
\label{cor1}
$\#\dD^m_n=\#\rR^m_n=\frac{1}{n+1}\binom{m(n+1)+n-1}{n}$. 
\end{corollary}

\begin{proof} 
From \cite[Lemma 4.1]{ath-gcn-04}  a  region in ${\rm Cat}^m (\Phi)$ is bounded if and only if it has 
no separating wall of type $ H_{\alpha, m}$ where $\alpha$ is a simple root.  
Let $\alpha_i$ be a simple root of type $A_n$. 
In view of Theorem \ref{all partitions}, the regions with no separating wall of type 
$H_{\alpha_i,m}$ biject to partitions $(\lambda_1,\lambda_2,\dots,\lambda_n)$ in $\pP^m_n$ 
for which $\lambda_i<m(n-i+1)$. On the other hand, in view of Theorem \ref{snakedisA}, 
the facets not containing the root $-\alpha_i$, biject to the same partitions as well. 
These partitions can be viewed as paths from $(0,0)$ to $(mn,n)$ which never touch the line $y=\frac{1}{m}x-1$, 
or equivalently, as paths from $(0,0)$ to $(n,m(n+1))$ 
which never touch the line $y=mx$ after the point $(0,0)$. 
From \cite[Theorem 3]{gs-msrp-03} 
we deduce that there are $\frac{1}{n+1}\binom{m(n+1)+n-1}{n}$ such paths. 
\end{proof} 

\subsection{Ongoing work}
In this paper we deal with the problem of finding a bijection between the set 
of dominant regions in the $m$-Catalan arrangement ${\rm Cat}^m(A_n)$ and that 
of facets of the $m$-generalized cluster complex $\Delta^m(A_n)$, where $m\in\mathbb{N}$. 
We further require the bijection to satisfy the property stated in Proposition \ref{ourproperty}. 
We answer this problem by providing a bijection which consists of two parts, 
where as intermediate step we use a certain set of integer  partitions. 
Moreover, we use these integer partitions for enumerating the bounded regions of ${\rm Cat}^m(A_n)$ 
and the facets of $\Delta^m_+(A_n)$.  

In \cite{fkt-12} we focus on types $B_n$ and $C_n$. So far we are able to characterize the set of facets 
of the generalized cluster complex $\Delta^m(B_n)$ and $\Delta^m(C_n)$ 
containing the negative simple root $-\alpha$, 
it terms of integers partitions. In particular we give a bijection between these sets of facets and the set 
$\pP^m(B_n)$ of partitions $(\lambda_1,\lambda_2,\dots,\lambda_n)$ for which $0\leq \lambda_i\leq mn$. 
The construction of the bijections from the set of dominant regions of 
the arrangements ${\rm Cat}^m(B_n)$ and ${\rm Cat}^m(C_n)$ to the set $\pP^m(B_n)$ is still in progress.

\subsection{Question}
Let $\Phi$ be a finite crystallographic root system. 
It would be very interesting to find a uniform bijection from the set $\dD^m(\Phi)$ of facets 
of the $m$-generalized cluster complex, to the set $\rR^m(\Phi)$ of dominant regions in the $m$-extended Catalan arrangement, 
which satisfies the property of Proposition \ref{ourproperty}.

\

\subsection*{Acknowledments}
We are grateful to Philippe Nadeau for helpful discussions and and Allesandro Conflitti for providing us the Formula (\ref{kij}). 
S. Fishel was  was partially supported by Simons Foundation grant  no. 209806 and NSF grant  no. 1200280.
M. Kallipoliti was funded by the FWF research grant no. Z130-N13. 
\bibliographystyle{plain}

\bibliography{mybibliography}

\end{document}